\documentclass[reqno,12pt]{amsart}
\usepackage{amsmath, amssymb, amsthm, amsfonts} 
\usepackage[english]{babel}
\usepackage{bbm}
\usepackage{graphicx}
\usepackage{subcaption} 
\usepackage{url}
\usepackage{soul}
\usepackage{epstopdf}
\usepackage[ruled,vlined]{algorithm2e}
\usepackage[a4paper,bindingoffset=0.5cm,left=2cm,right=2cm,top=2.5cm,bottom=2cm,footskip=.8cm]{geometry}
\usepackage{rotating}
\usepackage{amsbsy,enumerate,enumitem}
\usepackage{graphicx}
\usepackage{ccaption}
\usepackage{comment}
\usepackage{mathrsfs}
\usepackage{hyperref}
\usepackage{geometry}

\newcommand{\bs}{\boldsymbol}

\newcommand{\vb}{\vspace{3.2mm}}
\renewcommand{\hat}{\widehat}

\newcommand{\vertiii}[1]{{\left\vert\kern-0.25ex\left\vert\kern-0.25ex\left\vert #1 \right\vert\kern-0.25ex\right\vert\kern-0.25ex\right\vert}}

\allowdisplaybreaks

\newcommand{\queueingthreshold}{v}

\allowdisplaybreaks

\newtheorem{lemma}{Lemma}
\newtheorem{corollary}{Corollary}
\newtheorem{assumption}{Assumption}
\newtheorem{theorem}{Theorem}
\newtheorem{remark}{Remark}

\newtheorem{proposition}{Proposition}

\usepackage[utf8]{inputenc}
\usepackage{type1cm}         
\usepackage{multicol}        
\usepackage[bottom]{footmisc}

\usepackage{newtxtext}       %
\usepackage{newtxmath}       

\usepackage{pgfplots}

\usepackage{tikz}
\usepackage[mode=buildnew]{standalone}
\usetikzlibrary{automata,positioning}
\usetikzlibrary{decorations.pathreplacing,decorations.markings,shapes.geometric}
\usetikzlibrary{shapes,arrows}
\usetikzlibrary{backgrounds,calc,positioning}

\usepackage{pgfplots}
\pgfplotsset{compat=1.9}

\begin{document}

	\title[Transient behavior of threshold-dependent ruin and queueing models]{Transient Behavior of Threshold-Dependent Ruin\\ and Queueing Models with Phase-Type Jumps}
	\author[O. Boxma, M. Mandjes, D. Rutgers, {\tiny and} W. Scheinhardt]{Onno Boxma, Michel Mandjes, Dani\"el Rutgers {\tiny and} Werner Scheinhardt}
	
	\begin{abstract} 
		We study the transient behavior of two threshold-dependent stochastic models: a Cramér--Lundberg risk model and an M/G/1-type queueing model. In both models, the dynamics are driven by two different compound Poisson processes with drift, with the governing process depending on whether the current state is below or above a fixed threshold. For the risk model, we characterize the probability of ruin before an exponentially distributed epoch, while for the queueing model, we characterize the Laplace--Stieltjes transform of the workload at an exponentially distributed epoch. These quantities are characterized using fluctuation-theoretic results for spectrally positive Lévy processes, with Laplace transforms taken with respect to time for the risk model and with respect to both state and time for the queueing model. We consider general claim-size and job-size distributions and provide explicit representations of the relevant auxiliary quantities when these distributions are of phase-type.

		\vb
		
		
		\vb
		
		\noindent
		{\sc Affiliations.} OB is with Eurandom and Department of Mathematics and Computer Science, Eindhoven University of Technology, P.O. Box 513, 5600 MB Eindhoven, The Netherlands
		
		\noindent
		MM is with Mathematical Institute, Leiden University, P.O. Box 9512,
		2300 RA Leiden,
		The Netherlands. He is also affiliated with Korteweg-de Vries Institute for Mathematics, University of Amsterdam, Amsterdam, The Netherlands; E{\sc urandom}, Eindhoven University of Technology, Eindhoven, The Netherlands; Amsterdam Business School, Faculty of Economics and Business, University of Amsterdam, Amsterdam, The Netherlands. 
		
		\noindent
		DR is with Mathematical Institute, Leiden University, P.O. Box 9512,
		2300 RA Leiden,
		The Netherlands. 
		
		\noindent
		WS is with Faculty of Electrical Engineering, Mathematics and Computer Science, University of Twente, 
		P.O. Box 217, 7500 AE  Enschede, The Netherlands.
		
		\noindent Date: {\it \today}.
		
		\vb
		
		\noindent
		{\sc Acknowledgments.} 
		OB's and MM's research has been partly funded by the NWO Gravitation project N{\sc etworks}, grant number 024.002.003.

		\vb
		
		\noindent
		{\sc Email.} {\scriptsize \url{o.j.boxma@tue.nl}, \url{m.r.h.mandjes@math.leidenuniv.nl}, 
			\url{d.t.rutgers@math.leidenuniv.nl}, 
			\url{w.r.w.scheinhardt@utwente.nl}}.
		
		\vb 
		
		\noindent 
		{\sc Corresponding author.} DR

	\end{abstract}

	\maketitle
	
	\section{Introduction}\label{sec: intro}
	In this paper, we study the transient behavior of two threshold-dependent stochastic models: a Cram\'er--Lundberg risk model and an M/G/1-type queueing model. In both models, the dynamics change when the process crosses a prescribed threshold, with different compound Poisson processes with drift  (CPPs) governing the dynamics below and above the threshold. For the Cram\'er--Lundberg model, our objective is to characterize the probability of ruin before an arbitrary time $t>0$, while for the M/G/1-type queueing model, we aim to characterize the Laplace-Stieltjes transform (LST) of the workload at time $t$. We obtain these transient quantities by characterizing them at an independent exponentially distributed observation time $T_\beta$ with parameter $\beta>0$, so that their Laplace transforms with respect to time are given by $\beta^{-1}$ times the corresponding quantities evaluated at $T_\beta$.

	In our approach, we make use of fluctuation theory for L\'evy processes. In particular, we rely on results for spectrally positive L\'evy processes, that is, L\'evy processes with no negative jumps. These results provide useful tools for analyzing first-passage times, exit probabilities, and overshoot distributions for the compound Poisson processes governing the dynamics of our models. We, therefore, express several of the relevant quantities in terms of L\'evy-process characteristics, such as Laplace exponents and scale functions. {Our analysis combines standard fluctuation identities for spectrally positive L\'evy processes with structural properties specific to the Cram\'er--Lundberg ruin and M/G/1 queueing models. The latter enter at several key points of the analysis; see Section~\ref{sec:disc} for further discussion.}

	\subsubsection*{Models and objectives}
	Let $(Y_1(t))_{t\geqslant0}$ and $(Y_2(t))_{t\geqslant0}$ denote two CPPs with drift parameters $r_i<0$, arrival rates $\lambda_i$, and claim-sizes $B_i$, for $i=1,2$.
	
	\begin{itemize}
		\item[$\circ$]
		For the Cramér--Lundberg model, let $(Y_v(t))_{t\geqslant0}$ denote the
		total net cumulative claim process, with $Y_v(0)=0$, where
		$v\in(-\infty,u)$ is the threshold. The process $Y_v$ is driven by the
		CPP $Y_1$ while it resides below the threshold $v$, and by the CPP $Y_2$
		while it is above $v$. Our first objective is to evaluate the probability
		of ruin before an exponentially distributed epoch $T_\beta$ with parameter
		$\beta>0$. Let
		\[
		\overline{Y}_v(t):=\sup_{s\in[0,t]}Y_v(s)
		\]
		denote the running maximum of $Y_v$. We define
		\begin{equation}\label{eq: def ruin prob}
			p(u,v,\beta)
			:=\mathbb{P}\left(\overline{Y}_v(T_\beta)>u\right),
		\end{equation}
		and the all-time ruin probability by
		\begin{equation}\label{eq: def all time ruin prob}
			p(u,v)\equiv p(u,v,0)
			:=\mathbb{P}\left(\overline{Y}_v(\infty)>u\right).
		\end{equation}
		The time-dependent ruin probability can be evaluated via the inverse Laplace transform of \eqref{eq: def ruin prob} with respect to $\beta$.
		
		\item[$\circ$]The second model we consider is a dual CPP-driven queueing system. Let
		$(V_{\queueingthreshold}(t))_{t\geqslant0}$ denote the workload process,
		with arbitrary initial value
		$V_{\queueingthreshold}(0)=x$, for some $x\geqslant0$, and let
		$\queueingthreshold>0$ denote the threshold. When
		$V_{\queueingthreshold}$ resides below the threshold $\queueingthreshold$,
		it is driven by the CPP $Y_1$; equivalently, it behaves as the reflected
		version of $Y_1$ at zero. When it resides above $\queueingthreshold$, it is
		driven by $Y_2$ without reflection. For this model, our objective is to
		determine the LST of the workload at the
		exponential epoch $T_\beta$, given by
		\begin{equation}\label{eq: def LST workload}
			\chi(x)\equiv\chi(x,\queueingthreshold,\alpha,\beta)
			:=\mathbb{E}_x\left(
			e^{-\alpha V_{\queueingthreshold}(T_\beta)}
			\right)
			=
			\mathbb{E}\left(
			e^{-\alpha V_{\queueingthreshold}(T_\beta)}
			\,\middle|\,
			V_{\queueingthreshold}(0)=x
			\right).
		\end{equation}
		The LST in \eqref{eq: def LST workload} can also be interpreted as the double Laplace transform of the transient workload distribution, i.e.,
		\[
		\mathbb{E}\left(
		e^{-\alpha V_{\queueingthreshold}(T_\beta)}
		\,\middle|\,
		V_{\queueingthreshold}(0)=x
		\right) = \int_0^\infty \int_0^\infty {\rm e}^{-\alpha y}  \beta {\rm e}^{-\beta t} \mathbb{P}(V_v(t) \in \mathrm{d}y|V_v(0)=x)\,{\rm d} t,
		\]
		from which the transient density of the workload can be derived through Laplace inversion. For an algorithm that is specifically suited for the numerical inversion of double Laplace transforms we refer to \cite{dI}.
	\end{itemize}
	{Figure \ref{fig: example paths} illustrates sample paths of the two models.}
	
	\medskip
	
	In this paper, we characterize the ruin probability and the LST of the workload in terms of Laplace transforms with respect to time, for general claim-size and job-size distributions. We then specialize the auxiliary quantities appearing in these transform representations to phase-type distributions, for which explicit expressions are available. 
	
	\begin{figure}[ht]
		\centering
		\begin{subfigure}{0.49\textwidth}
			\centering
			\includegraphics[width=\textwidth]{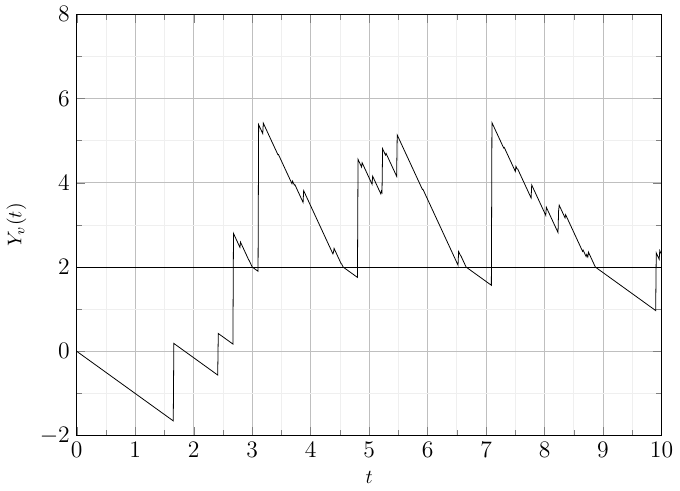}
			\caption{\label{fig: example path CL model}}
		\end{subfigure}
		\hfill
		\begin{subfigure}{0.49\textwidth}
			\centering
			\includegraphics[width=\textwidth]{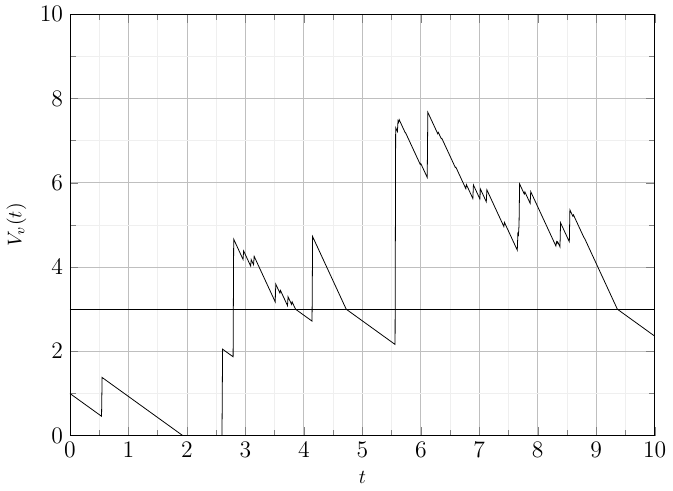}
			\caption{\label{fig: example path queueing model}}
		\end{subfigure}
		\caption{{\sc (A)} {(Cram\'er--Lundberg model)}: sample path of a net cumulative claim process over time with the threshold at $v=2$. {\sc (B)} {(queueing model)}: sample path of a workload process over time with the threshold at $v=3$ and initial level $x=1$.}
		\label{fig: example paths}
	\end{figure}

	\subsubsection*{Motivation} 
	The Cram\'er--Lundberg model and its variants play an important role in ruin theory and have numerous applications, including in insurance and credit risk management. While the standard Cram\'er--Lundberg model provides a basic framework for evaluating ruin probabilities, variants with level-dependent dynamics can provide a more flexible representation of an insurer's behavior. For example, an insurance company may adjust its premium rate or claims policy depending on its current reserve level. The model studied in this paper is one of the simpler examples of such a level-dependent risk model.
	Applications of the Cram\'er--Lundberg model often focus on the all-time ruin probability, as this is generally more tractable than the corresponding finite-time ruin probability. Although the all-time ruin probability provides a useful measure of long-term risk, the probability of ruin before a specified time horizon may be of greater practical interest. In our approach, the finite-time ruin probability is characterized through its Laplace transform with respect to time. The all-time ruin probability then follows as a limiting case of this characterization.
	
	The standard M/G/1 queue is a fundamental model for systems in which jobs arrive according to a Poisson process and require generally distributed service times, with applications including service systems, production and inventory systems, data centers, and healthcare. Related workload and storage models, such as fluid and dam models, provide analogous frameworks for systems in which accumulated workload or content evolves continuously over time. In many such settings, the dynamics of the system depend on the current workload or content level. For example, additional servers may be activated when the workload becomes large, or the input rate may be reduced to prevent excessive congestion. Queueing models with different dynamics above and below a specified threshold provide a natural framework for representing such state-dependent behavior and constitute a natural extension of the standard M/G/1 model.
	
	The queueing model considered in this paper focuses on the transient behavior of a threshold-dependent queue. While threshold-dependent queueing models have been studied extensively in steady-state, their transient behavior has received comparatively less attention. Transient analysis is relevant in settings where systems operate over finite periods, are shut down and reinitialized, or where the workload at a particular time is of interest.
	
	The ruin and queueing models considered in this paper are driven by compound Poisson processes with phase-type jump-size distributions. Phase-type distributions are particularly useful in this setting because they can approximate any distribution on $[0,\infty)$ arbitrarily closely; see, for example, \cite[Theorem III.4.2]{ASM}. Moreover, their tractability allows the auxiliary quantities appearing in our transform representations to be characterized explicitly. The general transform characterizations developed in this paper, however, do not require the jump-size distributions to be phase-type.
	
	\subsubsection*{Related literature} 
	We give a brief overview of the literature on level-dependent dynamic risk and queueing models, with an emphasis on contributions most closely related to the models studied in this paper.
	
	The predominant focus in the related literature is on the all-time ruin probability. The Cram\'er--Lundberg model in which only the premium rate $r$ depends on a fixed threshold has been studied, for example, in \cite[Section VIII.1a]{ASM2}; see also \cite{GS, LP, WAN} for related dividend problems. An extension to a general L\'evy setting with an adaptive linear drift is given in \cite{KL}. The Cram\'er-Lundberg model with multiple thresholds is studied in detail in \cite{AH} and \cite[Section VIII.1b]{ASM2}. In \cite{BJRS, BJRS2}, a slightly different setting is considered, in which the premium rate is adapted based on the level of the cumulative claim process. The adaptation takes place at claim arrival times in \cite{BJRS} and at Poisson observation times in \cite{BJRS2}. 
	
	The closely related M/G/1-type queueing model has also been studied extensively, although mostly in the steady-state setting. Early work in this area considers an exceptional first service after a busy period \cite{WEL} and the optimal switching level for the service rate \cite{COH2}.  Queueing systems with adaptable service rates are studied in \cite{BB, BBR2}, where the service rate can only be changed at specific instances, such as after arrivals or at Poisson observation times. A model with an adaptable arrival rate that depends on the workload level is studied in \cite{vOB}.
	
	The literature on level-dependent, L\'evy-driven queues is more limited. Steady-state analyses are given in \cite{BEK,BBR}, where the dynamics of the L\'evy processes change when a certain threshold is crossed or at Poisson moments. The model in \cite{BEK} additionally allows for a different, lower threshold for down-crossings. In \cite{ABKM}, both the transient and steady-state behavior of a L\'evy-driven queue are studied. In this model, the dynamics change when a threshold is crossed and only switch back after the end of the busy period. For a broad overview of level-dependent queueing systems, we refer to \cite{DSH}.
	
	For the standard background on phase-type distributions and L\'evy processes needed to follow this paper, we refer to \cite{BN} and \cite{KYP}, respectively.

	\subsubsection*{Organization} This paper is organized as follows. In Section \ref{sec: prelim} we define a number of auxiliary objects for CPPs with drift and give an overview of their properties which follow from L\'evy fluctuation theory. We also introduce in this section the notation for phase-type variables used throughout this paper, as well as some useful properties. In Section \ref{sec: aux for phase-type} we derive closed-form expressions for the objects defined in Section \ref{sec: prelim} for CPPs with phase-type distributed jumps. Section \ref{sec: Cramer-Lundberg model} then analyzes the ruin probability of the Cram\'er-Lundberg model with a threshold, and Section \ref{sec: queueing model} the LST of the queueing model with threshold, for general compound Poisson processes with drift. Finally, in Section~\ref{sec: expansions} we discuss a number of extensions and variations of the queueing model.

	\section{Preliminaries}\label{sec: prelim}
	In this section we collect some properties of the standard Cram\'er--Lundberg
	model that will be used throughout the paper. Let $(Y(t))_{t\geqslant0}$ be a
	compound Poisson process with negative drift $r<0$, arrival rate
	$\lambda>0$, and i.i.d.\ jump sizes distributed as $B$, that is,
	\[
	Y(t)=rt+\sum_{i=1}^{N(t)}B_i,
	\qquad t\geqslant0,
	\]
	where $N$ is a Poisson process with rate $\lambda$ and $Y(0)=0$. The
	corresponding workload process is obtained by reflection of $Y$ at zero;
	in this context, $Y$ is often referred to as the \emph{free process}.
	
	The process $Y$ is a spectrally positive L\'evy process, that is, a process
	with stationary and independent increments and no negative jumps
	\cite{KYP}. 
	We first recall some standard results from L\'evy fluctuation theory that will be used below. 
	At the end of this section we introduce the notation used for phase-type distributions throughout the paper.

	\subsection{L\'evy results}For a spectrally positive L\'evy process, it is convenient to work with its
	Laplace exponent
	\[
	\varphi(\alpha):=\log\mathbb{E}e^{-\alpha Y(1)}.
	\]
	For the compound Poisson process $Y$,
	\begin{equation}\label{eq: def laplace exp}
		\varphi(\alpha)
		=-r\alpha-\lambda\bigl(1-\mathscr{B}(\alpha)\bigr),
	\end{equation}
	where
	$
	\mathscr{B}(\alpha):=\mathbb{E}e^{-\alpha B}$
	is the LST of the jump-size distribution. We denote the right-inverse of
	$\varphi$ by $\psi$.
	
	We will also make extensive use of the \emph{scale functions}. For
	$\beta\geqslant0$, the primary scale function $W^{(\beta)}$ is defined by
	\[
	\int_0^\infty e^{-\alpha x}W^{(\beta)}(x)\,\mathrm{d}x
	=
	\frac{1}{\varphi(\alpha)-\beta},
	\qquad \varphi(\alpha)>\beta,
	\]
	and the corresponding secondary scale function is
	\[
	Z^{(\beta)}(u)
	:=
	1+\beta\int_0^u W^{(\beta)}(x)\,\mathrm{d}x,
	\qquad u\geqslant0.
	\]
	For background on scale functions of spectrally one-sided L\'evy processes,
	see, e.g., \cite[\S VIII.3]{KYP}, \cite[\S VII.2]{BER}, and \cite{KUZ}.
	
	We will repeatedly use the stopping times
	\begin{align}
		\sigma(u_-)&:=\inf\{t\geqslant0:Y(t)<-u_-\},\\
		\tau(u_+)&:=\inf\{t\geqslant0:Y(t)>u_+\},
	\end{align}
	for $u_-,u_+\geqslant0$. Thus, $\sigma(u_-)$ and $\tau(u_+)$ are the first
	times at which $Y$ falls below $-u_-$ and exceeds $u_+$, respectively.
	Since $Y$ has no negative jumps, it reaches $-u_-$ continuously, so that
	the condition $Y(t)\leqslant-u_-$ is equivalent to the one above.
	
	Finally, for $u_-+u_+>0$ and $\beta\geqslant0$, define
	\begin{align}
		\delta_-(u_-,u_+,\beta)
		&:=
		\mathbb{P}\bigl(
		\sigma(u_-)\leqslant
		\min\{\tau(u_+),T_\beta\}
		\bigr),\label{eq: def delta_min}\\
		\delta_+(u_-,u_+,\beta)
		&:=
		\mathbb{P}\bigl(
		\tau(u_+)\leqslant
		\min\{\sigma(u_-),T_\beta\}
		\bigr).\label{eq: def delta_plus}
	\end{align}
	Here $T_\beta$ denotes an independent exponential time with rate $\beta$,
	with $T_0=\infty$. Thus, $\delta_-$ and $\delta_+$ are the probabilities
	that $Y$ first reaches the lower or upper boundary, respectively, before
	reaching the other boundary and before killing. For $\beta=0$, this reduces
	to the usual two-sided hitting probabilities over an infinite time horizon.
	
	\medskip
	
	The following standard result expresses the two-sided hitting probabilities
	in terms of the scale functions.
	
	\begin{lemma}\label{lemma: deltas}
		For $u_-,u_+\geqslant0$ with $u_-+u_+>0$ and $\beta\geqslant0$,
		\begin{align}
			\delta_-(u_-,u_+,\beta)
			&=
			\frac{W^{(\beta)}(u_+)}
			{W^{(\beta)}(u_-+u_+)},
			\label{eq: delta_min}\\
			\delta_+(u_-,u_+,\beta)
			&=
			Z^{(\beta)}(u_+)
			-
			Z^{(\beta)}(u_-+u_+)
			\frac{W^{(\beta)}(u_+)}
			{W^{(\beta)}(u_-+u_+)},
			\label{eq: delta_plus}
		\end{align}
	\end{lemma}
	
	\begin{proof}
		These identities follow from the standard two-sided exit identities for
		spectrally one-sided L\'evy processes; see, e.g.,
		\cite[Theorem 8.1(iii)]{KYP}. The result there is stated for spectrally negative
		L\'evy processes and arbitrary initial states, so applying it to $-Y$ and
		shifting the boundaries gives \eqref{eq: delta_min}--\eqref{eq: delta_plus}.
	\end{proof}

	The following identities are also standard consequences of the
	one-sided exit identities for spectrally one-sided L\'evy processes; see,
	e.g., \cite[Theorem 8.1(ii)]{KYP} and \cite[Equation (8.2)]{KYP}.
	
	\begin{lemma}\label{lemma: Levy exit probabilities}
		For $u_-, u_+\geqslant0$ and $\beta\geqslant0$,
		\begin{align}
			\mathbb{P}\bigl(\sigma(u_-)\leqslant T_\beta\bigr)
			&=
			e^{-\psi(\beta)u_-}, \label{eq: sigma before killing}\\
			\mathbb{P}\bigl(\tau(u_+)\leqslant T_\beta\bigr)
			&=
			Z^{(\beta)}(u_+)
			-
			\frac{\beta}{\psi(\beta)}W^{(\beta)}(u_+).
			\label{eq: tau before killing}
		\end{align}
	\end{lemma}
	
	\begin{proof}
		The first identity follows immediately from the Wiener--Hopf
		decomposition; see, e.g., \cite[Equation (8.2)]{KYP}. The second is the
		standard one-sided exit identity in \cite[Theorem 8.1(ii)]{KYP}. In both
		cases, we apply the corresponding results for spectrally negative
		L\'evy processes to $-Y$, with the appropriate adjustment of the
		boundaries.
	\end{proof}
	
	In our analysis, we also require the distribution, in terms of its LST, of the overshoot of $Y$ over a level $u_+\geqslant0$, both on the event that $Y$ crosses $u_+$ before killing and on the event that it crosses $u_+$ before either killing or crossing the lower level $-u_-$. To this end, for $u_-,u_+\geqslant0$ with
	$u_-+u_+>0$, define
	\begin{align}
		\eta(u_+,\alpha,\beta)
		&:=
		\mathbb{E}\left[
		e^{-\alpha (Y(\tau(u_+))-u_+)}
		\mathbf{1}_{\{\tau(u_+)\leqslant T_\beta\}}
		\right],
		\label{eq:eta-def}\\
		\tilde\eta(u_-,u_+,\alpha,\beta)
		&:=
		\mathbb{E}\left[
		e^{-\alpha (Y(\tau(u_+))-u_+)}
		\mathbf{1}_{\{\tau(u_+)\leqslant
			\min\{\sigma(u_-),T_\beta\}\}}
		\right].
		\label{eq:eta-tilde-def}
	\end{align}
	Thus, $\eta$ and $\tilde\eta$ are the Laplace transforms of the
	corresponding (possibly defective) overshoot distributions. We denote
	their densities, when they exist, by $h(u_+,y,\beta)$ and
	$\tilde h(u_-,u_+,y,\beta)$, respectively. Below we give a closed-form
	expression for $\eta$ in terms of the primary scale function, the
	Laplace exponent, and its right-inverse, and show that
	$\tilde\eta$ admits a particularly simple representation in terms of
	$\eta$.

	\begin{lemma}\label{lemma: eta}
		For $u_+\geqslant0$ and $\alpha,\beta\geqslant0$,
		\begin{equation}\label{eq: eta}
			\eta(u_+,\alpha,\beta)
			= 
			(\varphi(\alpha)-\beta)
			\int_{u_+}^\infty e^{-\alpha (y-u_+)}W^{(\beta)}(y)\,\mathrm{d}y
			+
			\frac{\varphi(\alpha)-\beta}
			{\psi(\beta)-\alpha}
			W^{(\beta)}(u_+).
		\end{equation}
	\end{lemma}

	\begin{proof}
		The expression follows from \cite[Lemma 6]{MRS} with some minor modifications; see also
		\cite[Corollary 4]{IP}. The latter reference treats spectrally negative
		Markov additive processes and uses implicit killing, so the expression
		there requires the corresponding sign and killing-rate adjustments.
	\end{proof}
	
	The two-sided overshoot transform can be obtained directly from the
	one-sided transform.
	
	\begin{lemma}\label{lemma: eta_tilde}
		For $u_-,u_+\geqslant0$ with $u_-+u_+>0$, $\alpha\geqslant0$ and
		$\beta>0$,
		\begin{equation}\label{eq: eta_tilde}
			\tilde\eta(u_-,u_+,\alpha,\beta)
			=
			\eta(u_+,\alpha,\beta)
			-
			\delta_-(u_-,u_+,\beta)
			\eta(u_-+u_+,\alpha,\beta).
		\end{equation}
	\end{lemma}
	
	\begin{proof}
		Decompose according to whether the lower boundary is reached before
		$u_+$:
		\begin{align*}
			\tilde\eta(u_-,u_+,\alpha,\beta)
			&=
			\eta(u_+,\alpha,\beta)-
			\mathbb{E}\left[
			e^{-\alpha(Y(\tau(u_+))-u_+)}
			\mathbf{1}_{\{\sigma(u_-)\leqslant\tau(u_+)\leqslant T_\beta\}}
			\right].
		\end{align*}
		On the event in the second term, $Y$ reaches $-u_-$ before $u_+$ and,
		by the strong Markov property, the process after this time is distributed
		as a fresh copy of $Y$ started from $-u_-$. Shifting the process by
		$u_-+u_+$, the remaining overshoot problem is therefore described by
		$\eta(u_-+u_+,\alpha,\beta)$. Hence
		\[
		\mathbb{E}\left[
		e^{-\alpha(Y(\tau(u_+))-u_+)}
		\mathbf{1}_{\{\sigma(u_-)\leqslant\tau(u_+)\leqslant T_\beta\}}
		\right]
		=
		\delta_-(u_-,u_+,\beta)
		\eta(u_-+u_+,\alpha,\beta),
		\]
		which proves \eqref{eq: eta_tilde}.
	\end{proof}
	\subsection{Phase-type set-up}
	
	We use phase-type distributions throughout the remainder of the paper.
	Let $B$ be a phase-type random variable with $d\geqslant1$ transient
	phases, initial distribution $\bs a$ and subgenerator $\bs S$, and let
	$
	\bs s:=-\bs S\bs 1.
	$
	Thus, $a_i\geqslant0$ and $\bs a^\top\bs 1=1$. We write
	\[
	B\sim\mathbb{P}{\rm h}_d(\bs a,\bs S).
	\]
	The following standard identities will be used repeatedly; see
	\cite[Proposition III.4.1]{ASM2}.
	
	\begin{lemma}\label{lemma: dens and LST phase-type}
		For $B\sim\mathbb{P}{\rm h}_d(\bs a,\bs S)$, with
		$\bs s=-\bs S\bs 1$, its density, distribution function, and LST are
		\[
		f_B(y)=\bs a^\top e^{\bs S y}\bs s,\qquad
		F_B(y)=1-\bs a^\top e^{\bs S y}\bs 1,\qquad
		\mathscr B(\alpha)=\bs a^\top(\alpha\bs I-\bs S)^{-1}\bs s,
		\]
		for $y>0$, $y\geqslant0$, and $\alpha\geqslant0$, respectively.
	\end{lemma}

	For phase-type jumps, the Laplace exponent of $Y$ therefore becomes
	\begin{equation*}
		\varphi(\alpha)
		=
		-r\alpha
		-\lambda\left(
		1-\bs a^\top(\alpha\bs I-\bs S)^{-1}\bs s
		\right).
	\end{equation*}
	Consequently, $\varphi(\alpha)$ can be reduced to a fraction of polynomials in $\alpha$, of which the numerator is of degree $d+1$, 
	and therefore $\varphi(\alpha)=\beta$ reduces to a polynomial equation in $\alpha$ of degree $d+1$. For $\beta>0$, its roots have the following location.
	
	\begin{lemma}\label{lemma: roots phi = beta}
		For $\beta>0$, the equation $\varphi(\alpha)=\beta$ has $d$ roots in
		$\{\alpha\in\mathbb C:\Re(\alpha)<0\}$ and one root in
		$\{\alpha\in\mathbb C:\Re(\alpha)>0\}$.
	\end{lemma}
	
	\begin{proof}
		See, e.g., \cite[\S 3.1]{IBM}.
	\end{proof}
	
	\begin{assumption}\label{assump: distinct roots}
		For $\beta>0$, all roots of $\varphi(\alpha)=\beta$ are distinct.
	\end{assumption}
	
	This assumption is imposed throughout the remainder of the paper. It is
	a standard genericity condition in the analysis of phase-type models:
	multiple roots arise only for exceptional parameter configurations and
	are therefore avoided under an arbitrarily small perturbation of the
	model parameters. See also the discussion (just above Section 3.2) in \cite{ASM}. The assumption substantially simplifies the
	expressions below; repeated roots can in principle be handled by the
	usual limiting arguments, at the cost of additional polynomial factors
	in the resulting expressions.
	
	\section{Auxiliary objects for phase-type jumps}\label{sec: aux for phase-type}
	
	In this section we derive the auxiliary objects introduced in
	Section~\ref{sec: prelim} for CPPs with phase-type jumps. Throughout,
	$B\sim{\mathbb P}{\rm h}_{d}(\bs a,\bs S)$, where
	$\bs a\in[0,1]^d$ is the initial distribution and $\bs S$ is the
	generator matrix, with exit vector $\bs s:=-\bs S\bs 1$.
	
	By Lemma~\ref{lemma: roots phi = beta}, $\varphi(\alpha)-\beta$ has
	$d+1$ roots. Denote them by
	$\alpha_1(\beta),\ldots,\alpha_{d+1}(\beta)$, with
	$\alpha_{d+1}(\beta)=\psi(\beta)$ the unique root with positive real part,
	and write $\alpha_k\equiv\alpha_k(\beta)$. Under
	Assumption~\ref{assump: distinct roots}, partial fractions give
	\[
	\frac{1}{\varphi(\alpha)-\beta}
	=
	\sum_{k=1}^{d+1}\frac{c_k}{\alpha-\alpha_k},
	\]
	for suitable constants $c_k=c_k(\beta)$. Hence,
	\begin{equation}\label{eq: W_y phase type}
		W^{(\beta)}(y)=\sum_{k=1}^{d+1}c_ke^{\alpha_k y}.
	\end{equation}
	Moreover, since
	$\sum_{k=1}^{d+1}c_k/\alpha_k=\beta^{-1}$, we obtain
	\begin{equation}\label{eq: Z_u phase type}
		Z^{(\beta)}(u)
		=
		1+\beta\sum_{k=1}^{d+1}\frac{c_k}{\alpha_k}
		\left(e^{\alpha_k u}-1\right)
		=
		\beta\sum_{k=1}^{d+1}\frac{c_k}{\alpha_k}e^{\alpha_k u}.
	\end{equation}
	Below we use these representations to derive the overshoot transforms
	$\eta(u_+,\alpha,\beta)$ and
	$\tilde\eta(u_-,u_+,\alpha,\beta)$, their densities, and the associated
	hitting probabilities $\delta_-(u_-,u_+,\beta)$ and
	$\delta_+(u_-,u_+,\beta)$. We present two derivations. The first,
	based on the scale functions, follows directly by substituting
	\eqref{eq: W_y phase type} into the expressions for $\eta$ and
	$\tilde\eta$ in Lemmas~\ref{lemma: eta} and~\ref{lemma: eta_tilde}.
	The second follows the approach of \cite{ASM}, using the Kella--Whitt
	martingale to obtain systems of linear equations whose solutions yield
	the overshoot transforms and their densities.
	
	Both approaches have advantages. The scale-function approach is
	particularly simple, requiring only the roots of the Cram\'er--Lundberg
	equation, but yields the overshoot densities only implicitly through
	Laplace inversion. The martingale approach requires a matrix inversion,
	which can become computationally demanding for phase-type distributions
	with many phases, but provides explicit expressions for the overshoot
	densities in terms of the model parameters and thus avoids numerical
	Laplace inversion. All identified objects are used in the next sections.

	\subsection{Overshoot {above \texorpdfstring{$u_+$}{}} before killing}\label{sec: overshoot phase type}
	In this subsection we derive the transform of the overshoot
	$\eta(u_+,\alpha,\beta)$ and characterize the corresponding density
	$h(u_+,y,\beta)$ for phase-type distributed jumps using two
	complementary approaches.
	
	\subsubsection*{Scale functions}
	
	With the identities
	\[
	\sum_{k=1}^{d+1}\frac{c_k}{\alpha-\alpha_k}
	=\frac{1}{\varphi(\alpha)-\beta}
	\qquad\text{and}\qquad
	\psi(\beta)=\alpha_{d+1},
	\]
	substituting \eqref{eq: W_y phase type} into \eqref{eq: eta} and
	evaluating the integral gives
	\begin{equation}\label{eq: eta scale functions}
		\eta(u_+,\alpha,\beta)
		=
		(\varphi(\alpha)-\beta)
		\sum_{k=1}^{d}
		c_k
		\left(
		\frac{1}{\alpha-\alpha_k}
		-
		\frac{1}{\alpha-\alpha_{d+1}}
		\right)
		e^{\alpha_k u_+}.
	\end{equation}

	\subsubsection*{Martingales}
	
	Let $E_{u_+,j}$ denote the event that the level $u_+$ is crossed by a
	jump while that jump is in phase $j\in[d]$. This is well defined by
	viewing a phase-type jump as a sequence of exponential phases. Define
	\begin{equation}\label{eq: def p_i}
		\mathfrak{p}_j(u_+,\beta)
		:= \mathbb{P}(\tau(u_+)\leqslant T_\beta,E_{u_+,j})
		= \mathbb{E}\!\left[
		e^{-\beta\tau(u_+)}
		{\bs 1}{\{\tau(u_+)<\infty,E_{u_+,j}\}}
		\right],
	\end{equation}
	and let $\bs{\mathfrak p}(u_+,\beta)$ denote the corresponding
	$d$-dimensional vector. Following \cite[\S 4]{ASM}, the Kella--Whitt martingale yields,
	for the roots $\alpha_k$ with $\Re(\alpha_k)<0$,
	\begin{equation}\label{eq: sys eqs p_j}
		1 =
		e^{-\alpha_k u_+}
		\sum_{j=1}^d
		\mathfrak{p}_j(u_+,\beta)\mathscr{B}_j(\alpha_k),
		\qquad k=1,\ldots,d.
	\end{equation}
	
	Define
	\[
	\bs{\theta}(u_+,\beta)
	:=
	\bigl(e^{\alpha_1u_+},\ldots,e^{\alpha_du_+}\bigr)^\top,
	\qquad
	\bs B_{kj}(\beta):=\mathscr{B}_j(\alpha_k).
	\]
	Then \eqref{eq: sys eqs p_j} becomes
	$
	\bs B\,\bs{\mathfrak p}(u_+,\beta)
	=
	\bs{\theta}(u_+,\beta),$
	and hence, noting that by Lemma \ref{lemma:B invertible} in Appendix \ref{app: invertibility matrices} the matrix $\bs B$ is nonsingular,
	\begin{equation}
		\bs{\mathfrak p}(u_+,\beta)
		=\bs B^{-1}\bs{\theta}(u_+,\beta).
	\end{equation}
	
	We can now obtain the overshoot distribution by conditioning on the
	phase in which the crossing occurs. By the strong Markov property,
	\begin{align}\label{eq: overshoot I dens martingale}
		h(u_+,y,\beta)\,\mathrm{d}y
		&=
		\sum_{j=1}^d
		\mathbb{P}\!\left(
		Y(\tau(u_+))-u_+\in\mathrm{d}y
		\,\middle|\,
		\tau(u_+)\leqslant T_\beta,E_{u_+,j}
		\right)
		\mathfrak p_j(u_+,\beta) \notag\\
		&=
		\sum_{j=1}^d
		\left(\bs e_j^\top e^{\bs S y}\bs s\right)
		\mathfrak p_j(u_+,\beta)\,\mathrm{d}y \notag\\
		&=
		\bs{\mathfrak p}(u_+,\beta)^\top
		e^{\bs S y}\bs s\,\mathrm{d}y.
	\end{align}
	Thus, conditional on crossing before killing, the overshoot is again
	phase-type with generator $\bs S$ and initial distribution proportional
	to $\bs{\mathfrak p}(u_+,\beta)$. Its LST therefore follows directly
	from Lemma~\ref{lemma: dens and LST phase-type}:
	\begin{equation}\label{eq: eta martingales}
		\eta(u_+,\alpha,\beta)
		=
		\bs{\mathfrak p}(u_+,\beta)^\top
		(\alpha\bs I-\bs S)^{-1}\bs s.
	\end{equation}
	
	We use the martingale representation in the remainder of the paper,
	as it provides the overshoot density explicitly. In
	Appendix~\ref{app: equal etas}, we show that it agrees with the
	scale-function representation in \eqref{eq: eta scale functions}.

	\subsection{Overshoot {above \texorpdfstring{$u_+$}{} before killing or dropping below \texorpdfstring{$-u_-$}{}}}
	\label{sec: overshoot II phase type}
	
	We next derive the overshoot density
	$\tilde h(u_-,u_+,y,\beta)$ and its LST
	$\tilde\eta(u_-,u_+,\alpha,\beta)$, together with the hitting
	probabilities $\delta_-(u_-,u_+,\beta)$ and
	$\delta_+(u_-,u_+,\beta)$, for phase-type jumps. As before, we use
	two approaches: scale functions and the Kella--Whitt martingale.
	
	\subsubsection*{Scale functions} 
	
	Substituting \eqref{eq: W_y phase type} into \eqref{eq: delta_min} gives
	\begin{align}
		\delta_-(u_-,u_+,\beta)
		&=
		\frac{W^{(\beta)}(u_+)}
		{W^{(\beta)}(u_-+u_+)}=
		\frac{
			\sum_{\ell=1}^{d+1}c_\ell e^{\alpha_\ell u_+}
		}{
			\sum_{\ell=1}^{d+1}c_\ell e^{\alpha_\ell(u_-+u_+)}
		}.
		\label{eq: delta_min scale func}
	\end{align}
	Combining Lemma~\ref{lemma: eta_tilde} with
	\eqref{eq: eta scale functions} and \eqref{eq: delta_min scale func}
	gives
	\begin{align}\label{eq: eta tilde scale functions}
		\tilde\eta(u_-,u_+,\alpha,\beta)
		={}&(\varphi(\alpha)-\beta)
		\sum_{k=1}^{d}c_k
		\left(
		\frac{1}{\alpha-\alpha_k}
		-\frac{1}{\alpha-\alpha_{d+1}}
		\right)e^{\alpha_k u_+}
		\left(
		1-
		\frac{
			\sum_{\ell=1}^{d+1}c_\ell e^{\alpha_\ell u_+}
		}{
			\sum_{\ell=1}^{d+1}c_\ell e^{\alpha_\ell(u_-+u_+)}
		}
		e^{\alpha_k u_-}
		\right).
	\end{align}
	Substituting $\alpha=0$ in \eqref{eq: eta tilde scale functions} yields $\delta_+(u_-, u_+, \beta)$.

	\subsubsection*{Martingales}
	
	Let $E_{u_+,j}$ denote the event that the level $u_+$ is crossed
	in phase $j\in[d]$, and define
	\begin{equation}
		\mathfrak d_j(u_-,u_+,\beta)
		:=
		\mathbb P\left(
		\tau(u_+)\leqslant\min\{\sigma(u_-),T_\beta\},
		E_{u_+,j}
		\right).
	\end{equation}
	Let $\bs{\mathfrak d}(u_-,u_+,\beta)$ denote the corresponding vector.
	By the Kella--Whitt martingale argument of
	\cite[\S 3.4]{ASM}, for $k=1,\ldots,d+1$,
	\begin{equation}\label{eq: sys eqs delta_j}
		1
		=
		e^{\alpha_k u_-}\delta_-(u_-,u_+,\beta)
		+
		e^{-\alpha_k u_+}
		\sum_{j=1}^{d}
		\mathfrak d_j(u_-,u_+,\beta)\mathscr B_j(\alpha_k),
	\end{equation}
	where $\mathscr B_j(\cdot)$ denotes the LST of a phase-type variable
	with generator matrix $\bs S$ and initial phase $j$. Taking
	$u_-\to\infty$ recovers \eqref{eq: sys eqs p_j}.
	
	Using $\alpha_{d+1}=\psi(\beta)$ in \eqref{eq: sys eqs delta_j} gives
	\begin{equation}\label{eq: delta_min sys eqs}
		\delta_-(u_-,u_+,\beta)
		=
		e^{-\alpha_{d+1}u_-}
		-
		e^{-\alpha_{d+1}(u_-+u_+)}
		\sum_{j=1}^{d}
		\mathfrak d_j(u_-,u_+,\beta)
		\mathscr B_j(\alpha_{d+1}).
	\end{equation}
	Substituting this expression into \eqref{eq: sys eqs delta_j} and
	simplifying gives, for $k=1,\ldots,d$,
	\begin{align}\label{eq: temp sys eqs}
		&e^{\alpha_{d+1}(u_-+u_+)}e^{\alpha_k u_+}
		-e^{\alpha_{d+1}u_+}e^{\alpha_k(u_-+u_+)}
		\nonumber\\
		&\qquad=
		\sum_{j=1}^{d}\mathfrak d_j(u_-,u_+,\beta)
		\left[
		e^{\alpha_{d+1}(u_-+u_+)}\mathscr B_j(\alpha_k)
		-
		e^{\alpha_k(u_-+u_+)}\mathscr B_j(\alpha_{d+1})
		\right].
	\end{align}
	
	Let $\bs B$ be as defined in the previous subsection and set
	$
	\bs b_*:=
	\bigl(\mathscr B_1(\alpha_{d+1}),\ldots,
	\mathscr B_d(\alpha_{d+1})\bigr)^\top.
	$
	Define
	\[
	\bs M
	:=
	e^{\alpha_{d+1}(u_-+u_+)}\bs B
	-
	\begin{pmatrix}
		e^{\alpha_1(u_-+u_+)}\\
		\vdots\\
		e^{\alpha_d(u_-+u_+)}
	\end{pmatrix}
	\bs b_*^\top.
	\]
	Then \eqref{eq: temp sys eqs} can be written as
	\begin{equation}
		\bs{\mathfrak d}(u_-,u_+,\beta)
		=
		\bs M^{-1}
		\left[
		e^{\alpha_{d+1}(u_-+u_+)}
		\begin{pmatrix}
			e^{\alpha_1u_+}\\
			\vdots\\
			e^{\alpha_du_+}
		\end{pmatrix}
		-
		e^{\alpha_{d+1}u_+}
		\begin{pmatrix}
			e^{\alpha_1(u_-+u_+)}\\
			\vdots\\
			e^{\alpha_d(u_-+u_+)}
		\end{pmatrix}
		\right].
	\end{equation}
	
	As in the previous subsection, the overshoot is phase-type with
	generator $\bs S$ and initial vector $\bs{\mathfrak d}$. Hence
	\begin{align}
		\tilde h(u_-,u_+,y,\beta)
		&=
		\bs{\mathfrak d}(u_-,u_+,\beta)^\top
		e^{\bs S y}\bs s,
		\label{eq: h tilde vs frak d}\\
		\tilde\eta(u_-,u_+,\alpha,\beta)
		&=
		\bs{\mathfrak d}(u_-,u_+,\beta)^\top
		(\alpha\bs I-\bs S)^{-1}\bs s.
		\label{eq: eta tilde martingales}
	\end{align}
	Moreover,
	\begin{equation}\label{eq: delta plus vs frak d}
		\delta_+(u_-,u_+,\beta)
		=
		\sum_{j=1}^d\mathfrak d_j(u_-,u_+,\beta),
	\end{equation}
	while $\delta_-$ follows from \eqref{eq: delta_min sys eqs}.
	
	Throughout the remainder of the paper, we apply the above results to both
	$Y_1$ and $Y_2$. We indicate the underlying process by a subscript;
	for example, $\delta_{i,-}$ and $\delta_{i,+}$ denote the corresponding
	hitting probabilities for $Y_i$, $i\in\{1,2\}$. The same convention
	applies to the scale functions, stopping times, overshoot transforms,
	and densities.

	\section{Threshold-based Cram\'er-Lundberg model}\label{sec: Cramer-Lundberg model}
	In this section we focus on the first model considered in this paper, the level-dependent Cram\'er--Lundberg model. Recall that the total net cumulative claim process $Y_v$ is driven by the compound Poisson process $Y_1$ when it is below a threshold $v$, and by $Y_2$ when it is above $v$. Our goal is to evaluate the probability of ruin before an exponentially distributed epoch $T_\beta$ with parameter $\beta>0$, denoted by $p(u,v,\beta)$, as well as the ultimate ruin probability,
	\[
	p(u,v)\equiv p(u,v,0),
	\]
	for an initial surplus $u>\max\{0,v\}$.
	
	In Subsection \ref{sec: decomp}, we derive the ruin probability in terms of the scale functions introduced in Section~\ref{sec: prelim}. Subsection \ref{sec:numruin} provides numerical examples.

	\subsection{Decomposition}\label{sec: decomp}
	We characterize the ruin probability before killing and the all-time ruin
	probability by decomposing the event of ruin into distinct scenarios. The
	decomposition depends on whether the threshold $v$ is positive, negative,
	or zero. Since the case $v=0$ serves as a regenerative point for the other
	two cases, we consider it first. The main ideas underlying our
	decomposition follow those in \cite[\S 5.3]{MB}.

	\begin{remark}\label{remark: }
		{\em Throughout this section, in the derivation of the ruin probabilities we often consider the scenario `before killing', which in the case of the all-time ruin probability should be interpreted as `ever'. Specifically, we do not need a different derivation for the all-time ruin probability as the auxiliary objects introduced in Section \ref{sec: prelim} are well defined for $\beta=0$. However, since Assumption \ref{assump: distinct roots} does not always hold for $\beta=0$, the auxiliary objects can attain different forms than the ones derived in Section \ref{sec: aux for phase-type}.}
	\end{remark}
	
	\subsubsection*{Threshold at zero}
	We analyze the case $v=0$ first since it can be used as a regenerative point in the other two cases. In the following lemma we show that the ruin probability can be expressed in terms of the hitting probabilities and overshoot densities defined in Section \ref{sec: prelim}.
	
	\begin{lemma}\label{lemma: ruin prob v=0}
		For $u>0$ and $\beta\geqslant0$
		\begin{equation}\label{eq: ruin prob v=0}
			p(u,0,\beta) = \frac{\int_0^{u} h_1(0, y, \beta)\,\delta_{2,+}(y, u-y, \beta)\,\mathrm{d}y + \int_{u}^\infty h_1(0, y, \beta)\,\mathrm{d}y}{1 - \int_0^{u} h_1(0, y, \beta)\,\delta_{2,-}(y, u-y,\beta)\,\mathrm{d}y}.
		\end{equation}
	\end{lemma}
	\begin{proof}We decompose according to the first event after the process crosses the
		threshold $v=0$: either the overshoot already exceeds $u$, the process
		subsequently reaches $u$ before returning to the threshold, or it returns
		to the threshold before reaching $u$. By the strong Markov property,
		\begin{align*}
			p(u,0,\beta)
			&=
			\int_u^\infty h_1(0,y,\beta)\,\mathrm{d}y
			+\int_0^u h_1(0,y,\beta)\,
			\delta_{2,+}(y,u-y,\beta)\,\mathrm{d}y \\
			&\quad
			+\int_0^u h_1(0,y,\beta)\,
			\delta_{2,-}(y,u-y,\beta)\,\mathrm{d}y
			\,p(u,0,\beta).
		\end{align*}
		The three terms correspond, respectively, to ruin occurring directly upon
		crossing the threshold, ruin occurring before returning to the threshold,
		and returning to the threshold before ruin. Since
		\[
		0\leqslant
		\int_0^u h_1(0,y,\beta)\,
		\delta_{2,-}(y,u-y,\beta)\,\mathrm{d}y
		<1,
		\]
		we can solve for $p(u,0,\beta)$, yielding \eqref{eq: ruin prob v=0}.
	\end{proof}
	
	\subsubsection*{Positive-valued threshold}
	Here we analyze the case $v\in(0,u)$. This scenario can be interpreted as starting in the low-risk zone. In the following lemma we again give an expression of the ruin probability in terms of the hitting probabilities and the overshoot densities. 
	
	\begin{lemma}\label{lemma: ruin prob v>0}
		For $u>v>0$ and $\beta\geqslant0$
		\begin{align}\label{eq: ruin prob v>0}
			\notag p(u,v,\beta) &= \int_{u-v}^\infty h_1(v,y,\beta) \,\mathrm{d}y + \int_{0}^{u-v} \delta_{2,+}(y, u-v-y, \beta) h_1(v,y,\beta) \,\mathrm{d}y\\
			&\:\:\:\:\:\:\:+ \int_{0}^{u-v} \delta_{2,-}(y, u-v-y, \beta) h_1(v,y,\beta) \,\mathrm{d}y \times p(u-v, 0, \beta)
		\end{align}
		with $p(u-v, 0, \beta)$ as given in \eqref{eq: ruin prob v=0}.
	\end{lemma}
	\begin{proof}
		Similar to the proof of Lemma~\ref{lemma: ruin prob v=0}, we decompose
		according to the first event after the process crosses the threshold $v$.
		If the overshoot is at least $u-v$, ruin occurs immediately. Otherwise,
		the process either reaches $u$ before returning to $v$, or returns to $v$
		before reaching $u$. In the latter case, the strong Markov property reduces
		the problem to the one with threshold $0$ and initial surplus $u-v$.
		This yields the three terms in \eqref{eq: ruin prob v>0}.
	\end{proof}

	\subsubsection*{Negative-valued threshold} 
	The final case we consider is when the threshold $v$ is less than $0$, i.e., when the model is initialized in the `high-risk zone'. In this case we aim to analyze the ruin probability for a surplus $u\geqslant 0$. We specifically allow $u=0$ here since $v < 0$, in which case the ruin probability can be interpreted as the probability of the net cumulative claim process becoming positive and hence attaining a negative surplus.
	
	\begin{lemma}\label{lemma: ruin prob v<0}
		For $u \geqslant 0 > v$ and $\beta \geqslant 0$
		\begin{equation}\label{eq: ruin prob v<0}
			p(u,v,\beta) = \delta_{2,+}(-v,u,\beta) + \delta_{2,-}(-v, u, \beta) \times p(u-v,0,\beta)
		\end{equation}
		with $p(u-v,0,\beta)$ as given in \eqref{eq: ruin prob v=0}.
	\end{lemma}
	\begin{proof}
		For this case we only need to consider two scenarios: (i) $Y_v$ crosses the level $u$ before it reaches the threshold $v$ and before killing, and (ii) $Y_v$ reaches the threshold $v$ before killing and before crossing the level $u$, at which point, by the strong Markov property, we are in a scenario which is initialized at the threshold. These two scenarios correspond to the first and second term on the right-hand side of \eqref{eq: ruin prob v<0}. 
	\end{proof}

	\begin{remark}\label{rem: all-time ruin prob}
		{\em When $\mathbb{E}(Y_1(1)) > 0$ the all-time ruin probability is always equal to one. Indeed, since $\mathbb{P}(\tau_1(0) < \infty) = 1$ in this case, we have 
			\begin{align*}
				\int_0^{u} &h_1(0, y, 0)\,\delta_{2,+}(y, u-y, 0)\,\mathrm{d}y + \int_{u}^\infty h_1(0, y, 0)\,\mathrm{d}y \\
				=&\: \int_0^{u} h_1(0, y, 0)\,\left(1-\delta_{2,-}(y, u-y, 0)\right)\,\mathrm{d}y + \int_{u}^\infty h_1(0, y, 0)\,\mathrm{d}y \\
				=&\: 1 - \int_0^{u} h_1(0, y, 0)\,\delta_{2,-}(y, u-y, 0)\,\mathrm{d}y,
			\end{align*}
			so that $p(u,0,0) = 1$ by Lemma \ref{lemma: ruin prob v=0}. It is easily verified from Lemma \ref{lemma: ruin prob v>0} and Lemma \ref{lemma: ruin prob v<0} that $p(u,v,0) = 1$ for any $v < u$.}
	\end{remark}

	\subsection{Numerical example}\label{sec:numruin}
	In this subsection, we consider a specific instance of the risk model and present some numerical results for the ruin probabilities. We have obtained the all-time ruin probabilities from Lemmas \ref{lemma: ruin prob v=0}, \ref{lemma: ruin prob v>0} and \ref{lemma: ruin prob v<0} directly for $\beta=0$, while the time-dependent ruin probabilities follow from the same lemmas in combination with the numerical Laplace inversion algorithm from~\cite[\S~8]{AW2}.
	
	The CPPs $Y_i$, for $i=1,2$, that determine the dynamics of the instance each have a drift $r_i<0$, an arrival rate $\lambda_i>0$ and claim sizes that are Erlang distributed with parameters $k_i\in\mathbb{N}$ and $\mu_i > 0$. We denote each CPP $Y_i$ via a parameter tuple $(r_i, \lambda_i, k_i, \mu_i)$. For the specific instance considered in this subsection, we have used the following parameters, 
	\begin{equation*}
		(r_1, \lambda_1, k_1, \mu_1) = (-2, 1, 3, 2) \quad \text{and} \quad (r_2, \lambda_2, k_2, \mu_2) = (-3, \tfrac{1}{2}, 2, 1).
	\end{equation*}
	Therefore, the Laplace exponents are given by
	\begin{equation*}
		\varphi_1(\alpha) = 2\alpha - 1 + \frac{8}{(2+\alpha)^3} \quad \text{and} \quad \varphi_2(\alpha) = 3\alpha - \frac{1}{2} + \frac{1}{2(\alpha+1)^2}.
	\end{equation*}
	
	In Figure \ref{fig: CL model} we present ruin probabilities for the instance considered. Panel {\sc (a)} shows the ruin probability before time $t$ for fixed values of the surplus $u$ and the threshold $v$, with the dashed line denoting the all-time ruin probability. The solid lines represent the values obtained by numerical Laplace inversion, while the marks denote 95\% confidence intervals of Monte Carlo based estimates of $10^5$ samples. In Panel {\sc (b)} we have plotted the all-time ruin probability as function of the threshold for a surplus $u=10$. Furthermore, the dotted blue line and the dashed magenta line denote the all-time ruin probability of the processes $Y_1$ and $Y_2$ separately, respectively, which are obtained from \eqref{eq: tau before killing} for $\beta=0$. The marks denote 95\% confidence intervals of Monte Carlo based estimates, in this case based on 10$^4$ samples.
	
	\begin{figure}[ht]
		\centering
		\begin{subfigure}{0.49\textwidth}
			\includegraphics[width=\textwidth]{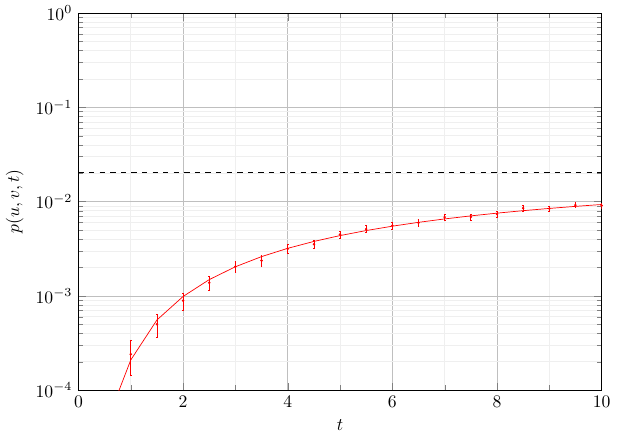}
			\caption{\label{fig: CL model t vs prob}}
		\end{subfigure}
		\hfill
		\begin{subfigure}{0.49\textwidth}
			\includegraphics[width=\textwidth]{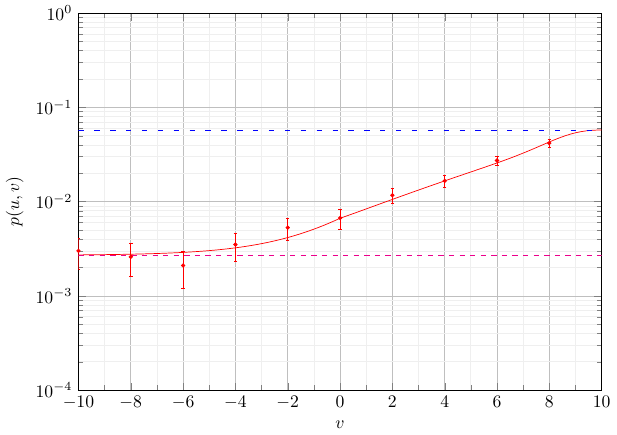}
			\caption{\label{fig: CL model v vs u}}
		\end{subfigure}
		\caption{\label{fig: CL model} Time-dependent ruin probability for $u=10$ and $v=5$ as a function of time, with the dashed gray line denoting the corresponding all-time ruin probability ({\sc a}), and the all-time ruin probability for $u=10$ as a function of $v$, with the loosely dashed blue line and the dashed magenta line denoting the all-time ruin probabilities of $Y_1$ and $Y_2$, respectively ({\sc b}).}
	\end{figure}
	
	\section{Threshold-based queueing}\label{sec: queueing model}
	In this section we study the workload process
	$(V_{\queueingthreshold}(t))_{t\geqslant0}$ of a queueing system that is
	driven by the CPP $Y_1$ when the workload is below the threshold
	$\queueingthreshold$, and by the CPP $Y_2$ when the workload is above
	$\queueingthreshold$. Our goal is to characterize the distribution of the
	workload at an exponential epoch through its Laplace--Stieltjes transform
	$\chi(x)$, as defined in \eqref{eq: def LST workload}.

	\subsection{Decomposition}
	
	To derive the LST of the workload, we introduce the following auxiliary
	quantity. For $u_-,u_+\geqslant0$ with $u_-+u_+>0$, $\alpha\geqslant0$,
	and $\beta>0$, define
	\begin{equation}\label{eq: delta_star def}
		\delta_{i,\star}(u_-,u_+)
		\equiv
		\delta_{i,\star}(u_-,u_+,\alpha,\beta)
		:=
		e^{-\alpha u_-}
		\mathbb{E}\left[
		e^{-\alpha Y_i(T_\beta)}
		\mathbf{1}_{\{T_\beta<
			\min\{\sigma_i(u_-),\tau_i(u_+)\}\}}
		\right].
	\end{equation}
	Thus, $\delta_{i,\star}(u_-,u_+)$ is the contribution to the workload
	LST from killing before the free process $Y_i$, started at $u_-$, reaches
	either $0$ or $u_-+u_+$.
	The following result, adapted from \cite[Lemma 9]{MRS}, expresses
	$\delta_{i,\star}$ in terms of the auxiliary objects introduced in
	Section~\ref{sec: prelim}.
	
	\begin{lemma}\label{lemma: delta_star}
		For $x\in[0,v]$, $i=1,2$, $\alpha\geqslant0$, and $\beta>0$,
		\begin{equation}\label{eq: delta_star decomp}
			\delta_{i,\star}(x,v-x)
			=
			\left(
			e^{-\alpha x}
			-\delta_{i,-}(x)
			-e^{-\alpha v}\eta_i(v-x,\alpha,\beta)
			+e^{-\alpha v}\delta_{i,-}(x)\eta_i(v,\alpha,\beta)
			\right)
			\frac{\beta}{\beta-\varphi_i(\alpha)}.
		\end{equation}
	\end{lemma}
	
	\begin{corollary}
		For $x\in[0,v]$, $i=1,2$, $\alpha\geqslant0$, and $\beta>0$,
		\begin{equation}\label{eq: delta_star scale func}
			\delta_{i,\star}(x,v-x)
			=
			\beta
			\frac{W_i^{(\beta)}(v-x)}
			{W_i^{(\beta)}(v)}
			\int_0^v e^{-\alpha y}W_i^{(\beta)}(y)\,\mathrm{d}y
			-
			\beta e^{-\alpha x}
			\int_0^{v-x}e^{-\alpha y}W_i^{(\beta)}(y)\,\mathrm{d}y.
		\end{equation}
	\end{corollary}
	
	\begin{proof}
		This follows by substituting the expressions for
		$\delta_{i,-}$ and $\eta_i$ from \eqref{eq: delta_min} and
		\eqref{eq: eta} into \eqref{eq: delta_star decomp}.
	\end{proof}
	
	We can now characterize the LST of the workload at the exponential epoch.
	The following proposition expresses $\chi(x)$ in terms of the hitting
	probabilities, overshoot transforms, and the auxiliary quantity
	$\delta_{1,\star}$ defined in \eqref{eq: delta_star def}.
	
	\begin{proposition}\label{prop: chi(x)}
		Let $\beta>0$. For $x\in(v,\infty)$
		\begin{equation}\label{eq: chi(x) for x > v}
			\chi(x)
			=
			\frac{\beta}{\beta-\varphi_2(\alpha)}
			\left(
			e^{-\alpha x}
			-e^{-\alpha v}e^{-\psi_2(\beta)(x-v)}
			\right)
			+
			e^{-\psi_2(\beta)(x-v)}\chi(v),
		\end{equation}
		and for $x\in[0,v]$
		\begin{align}
			\chi(x)
			&=
			\delta_{1,-}(x,v-x)\chi(0)
			+\tilde{\eta}_1(x,v-x,\psi_2(\beta),\beta)\chi(v)
			\notag\\
			&\quad
			+\frac{\beta}{\beta-\varphi_2(\alpha)}
			e^{-\alpha v}
			\left(
			\tilde{\eta}_1(x,v-x,\alpha,\beta)
			-
			\tilde{\eta}_1(x,v-x,\psi_2(\beta),\beta)
			\right)
			+\delta_{1,\star}(x,v-x).
			\label{eq: chi(x) for x < v}
		\end{align}
	\end{proposition}
	
	\begin{proof}
		For the first claim, condition on the first event, which is either killing or
		hitting the threshold $v$. In terms of the free process $Y_2$,
		\begin{align*}
			\chi(x)
			&=
			e^{-\alpha x}
			\mathbb{E}\left[
			e^{-\alpha Y_2(T_\beta)}
			\mathbf{1}_{\{T_\beta<\sigma_2(x-v)\}}
			\right]
			+
			\mathbb{P}\left(\sigma_2(x-v)\leqslant T_\beta\right)\chi(v).
		\end{align*}
		By \eqref{eq: sigma before killing} and the strong Markov
		property we have that
		$
		\mathbb{P}\left(\sigma_2(x-v)\leqslant T_\beta\right)
		=e^{-\psi_2(\beta)(x-v)},
		$
		while
		\[
		\mathbb{E}\left(e^{-\alpha Y_2(T_\beta)}\right)
		=
		\frac{\beta}{\beta-\varphi_2(\alpha)}.
		\]
		Hence,
		\begin{align*}
			\chi(x)
			&=
			e^{-\alpha x}
			\frac{\beta}{\beta-\varphi_2(\alpha)}
			-
			e^{-\alpha v}
			\frac{\beta}{\beta-\varphi_2(\alpha)}
			e^{-\psi_2(\beta)(x-v)}
			+
			e^{-\psi_2(\beta)(x-v)}\chi(v),
		\end{align*}
		which gives \eqref{eq: chi(x) for x > v}.
		
		For the second claim, conditioning on the first event for $Y_1$ (which is either reaching $0$, exceeding $v$, or being killed)
		\begin{equation*}
			\chi(x)
			=
			\delta_{1,-}(x,v-x)\chi(0)
			+
			\int_0^\infty
			\tilde h_1(x,v-x,y,\beta)\chi(v+y)\,\mathrm{d}y
			+
			\delta_{1,\star}(x,v-x).
		\end{equation*}
		Substituting \eqref{eq: chi(x) for x > v} into the integral gives
		\begin{align*}
			\int_0^\infty
			\tilde h_1(x,v-x,y,\beta)\chi(v+y)\,\mathrm{d}y
			&=
			\frac{\beta}{\beta-\varphi_2(\alpha)}
			e^{-\alpha v}
			\left(
			\tilde{\eta}_1(x,v-x,\alpha,\beta)
			-
			\tilde{\eta}_1(x,v-x,\psi_2(\beta),\beta)
			\right)
			\\
			&\quad
			+
			\tilde{\eta}_1(x,v-x,\psi_2(\beta),\beta)\chi(v),
		\end{align*}
		and \eqref{eq: chi(x) for x < v} follows.
	\end{proof}

	\subsection{Derivation of \texorpdfstring{$\chi(0)$}{chi 0} and \texorpdfstring{$\chi(v)$}{chi v}}\label{sec: chi 0 and chi K}
	
	We now characterize $\chi(0)$ and $\chi(v)$, which, together with Proposition \ref{prop: chi(x)}, determine $\chi(x)$ for all $x\geqslant0$. We obtain two linear equations by conditioning on the first event and applying the strong Markov property. This yields a linear system in $\chi(0)$ and $\chi(v)$, from which both quantities can be determined.
	
	Starting from $0$, the first event is either killing, an arrival from $Y_1$ with job size at most $v$, or an arrival with job size exceeding $v$. Hence,
	\begin{align*}
		\chi(0)
		&= \frac{\beta}{\lambda_1+\beta}
		+ \frac{\lambda_1}{\lambda_1+\beta}
		\left(
		\int_{[0,v]} \mathbb{P}(B_1\in\mathrm{d}x)\chi(x)
		+ \int_{(v,\infty)} \mathbb{P}(B_1\in\mathrm{d}x)\chi(x)
		\right).
	\end{align*}
	Substituting \eqref{eq: chi(x) for x < v} and \eqref{eq: chi(x) for x > v} into the respective integrals gives
	\begin{align*}
		\chi(0)
		&= P_{11}(v,\beta)\chi(0)
		+ P_{12}(v,\beta)\chi(v)
		+ b_1(v,\alpha,\beta),
	\end{align*}
	where
	\begin{align}
		P_{11}(v,\beta)
		&= \frac{\lambda_1}{\lambda_1+\beta}
		\int_{[0,v]}
		\delta_{1,-}(x,v-x)
		\mathbb{P}(B_1\in\mathrm{d}x),
		\label{eq: def P_{11}}\\
		P_{12}(v,\beta)
		&= \frac{\lambda_1}{\lambda_1+\beta}
		\left(
		\int_{[0,v]}
		\tilde{\eta}_1(x,v-x,\psi_2(\beta),\beta)
		\mathbb{P}(B_1\in\mathrm{d}x)
		+ \int_{(v,\infty)}
		e^{-\psi_2(\beta)(x-v)}
		\mathbb{P}(B_1\in\mathrm{d}x)
		\right),
		\label{eq: def P_{12}}\\
		\notag
		b_1(v,\alpha,\beta)
		&= \frac{\beta}{\lambda_1+\beta}
		+ \frac{\lambda_1}{\lambda_1+\beta}
		\Bigg(
		\frac{\beta}{\beta-\varphi_2(\alpha)}
		\int_{(v,\infty)}
		\left(
		e^{-\alpha x}
		- e^{-\alpha v}e^{-\psi_2(\beta)(x-v)}
		\right)
		\mathbb{P}(B_1\in\mathrm{d}x) \\
		\notag
		&\qquad\qquad
		+ \frac{\beta}{\beta-\varphi_2(\alpha)}
		e^{-\alpha v}
		\int_{[0,v]}
		\left(
		\tilde{\eta}_1(x,v-x,\alpha,\beta)
		- \tilde{\eta}_1(x,v-x,\psi_2(\beta),\beta)
		\right)
		\mathbb{P}(B_1\in\mathrm{d}x) \\
		&\qquad\qquad
		+ \int_{[0,v]}
		\delta_{1,\star}(x,v-x)
		\mathbb{P}(B_1\in\mathrm{d}x)
		\Bigg).
		\label{eq: def b_1}
	\end{align}
	
	Next, evaluating \eqref{eq: chi(x) for x < v} at $x=v$ gives
	\begin{align*}
		\chi(v)
		&= P_{21}(v,\beta)\chi(0)
		+ P_{22}(v,\beta)\chi(v)
		+ b_2(v,\alpha,\beta),
	\end{align*}
	where
	\begin{align}
		P_{21}(v,\beta)
		&= \delta_{1,-}(v,0),\quad\quad
		P_{22}(v,\beta)
		= \tilde{\eta}_1(v,0,\psi_2(\beta),\beta),\label{eq: def P_{2}}\\
		b_2(v,\alpha,\beta)
		&= \frac{\beta}{\beta-\varphi_2(\alpha)}
		e^{-\alpha v}
		\left(
		\tilde{\eta}_1(v,0,\alpha,\beta)
		- \tilde{\eta}_1(v,0,\psi_2(\beta),\beta)
		\right)
		+ \delta_{1,\star}(v,0). \label{eq: def b_2}
	\end{align}
	
	Combining the two equations gives
	\begin{equation*}
		\begin{pmatrix}
			\chi(0)\\
			\chi(v)
		\end{pmatrix}
		=
		\bs{P}
		\begin{pmatrix}
			\chi(0)\\
			\chi(v)
		\end{pmatrix}
		+ \bs{b},
	\end{equation*}
	where
	\begin{equation*}
		\bs{P}
		=
		\begin{pmatrix}
			P_{11}(v,\beta) & P_{12}(v,\beta)\\
			P_{21}(v,\beta) & P_{22}(v,\beta)
		\end{pmatrix},
		\qquad
		\bs{b}
		=
		\begin{pmatrix}
			b_1(v,\alpha,\beta)\\
			b_2(v,\alpha,\beta)
		\end{pmatrix}.
	\end{equation*}
	
	The row-sums of $\bs{P}$ are strictly less than one, so that $\bs{I}-\bs{P}$ is invertible. Consequently,
	\begin{equation}\label{eq: matrix equation chi(0) and chi(v)}
		\begin{pmatrix}
			\chi(0)\\
			\chi(v)
		\end{pmatrix}
		=
		\left(\bs{I}-\bs{P}\right)^{-1}\bs{b}.
	\end{equation}
	
	\begin{theorem}\label{theorem: transient LST chi(x)}
		For $x\geqslant0$ and $\alpha,\beta>0$, the LST of the workload at an exponential time with parameter $\beta$ is given by Proposition \ref{prop: chi(x)}, where $\chi(0)$ and $\chi(v)$ are determined by \eqref{eq: matrix equation chi(0) and chi(v)}.
	\end{theorem}

	\begin{remark}\label{remark: duality}
		{\em
			In general, the well-known duality between ruin probabilities and workload distributions does not directly apply to the model considered here. An exception occurs when $Y_1$ and $Y_2$ differ only in their drift parameters, i.e., when $\lambda_1=\lambda_2$, $B_1$ and $B_2$ are identically distributed, and $r_1\neq r_2$. In this case, we can express $\chi(0)$ in terms of a ruin probability.
			
			To see this, for $z\geqslant0$ define $Y_{z-v}^*$ as the cumulative claim process of Section \ref{sec: Cramer-Lundberg model} with threshold $v^*=z-v$, driven by $Y_2$ below $v^*$ and by $Y_1$ above $v^*$. Let $R^*(t)=z-Y_{z-v}^*(t)$. By the duality argument of \cite[\S III.2, Theorem 2.1]{AA}, the risk process obtained from the workload process has the same distribution as $R^*$. Consequently,
			\begin{equation}\label{eq: duality}
				{\mathbb{P}\left(V_v(T_\beta)>z\,\middle|\,V_v(0)=0\right)}
				=
				\mathbb{P}\left(\bar{Y}_{z-v}^*(T_\beta)>z\right).
			\end{equation}
			Taking the Laplace transform with respect to $z$ therefore yields
			\begin{equation*}
				\chi(0)
				=
				1-\alpha\int_0^\infty
				e^{-\alpha z}
				\mathbb{P}\left(\bar{Y}_{z-v}^*(T_\beta)>z\right)
				\,\mathrm{d}z,
			\end{equation*}
			where the ruin probability on the right-hand side is given by the results of Section \ref{sec: Cramer-Lundberg model}.
			\hfill$\Diamond$}
	\end{remark}

	\subsection{Stationary distribution}
	
	In Theorem \ref{theorem: transient LST chi(x)} we derived the transient LST of the workload at an exponentially distributed time $T_\beta$. We now consider the {\it stationary} workload, denoted by
	\[
	V_v \equiv V_v(\infty),
	\]
	and derive its LST. In \cite[Subsection 4.2]{BBR} this derivation was done by means of regeneration epochs initialized at the threshold level. Applying a different approach, in this subsection and Appendix \ref{app: proof equivalence LSTs} we show that our result is equal to the LST derived in \cite{BBR}.
	
	Since stationarity corresponds to $\beta=0$, we have
	\begin{equation}
		\mathbb{E}\,e^{-\alpha V_v}
		=
		\lim_{\beta\downarrow0}\chi(x,\alpha,\beta),
	\end{equation}
	for any $x\geqslant0$, since the stationary distribution is independent of the initial workload. We obtain the stationary LST by taking the limit $\beta\downarrow0$ in $\chi(0,\alpha,\beta)$.
	To ensure that a stationary distribution exists, throughout this section we impose the following stability condition on $Y_2$; no conditions are imposed on the process $Y_1$.
	
	\begin{assumption}\label{assump: stability}
		We assume that $-\mathbb{E}\,Y_2(1)=\varphi_2'(0)>0.$
	\end{assumption}
	
	Assumption \ref{assump: stability} implies the following standard properties; see e.g.\ \cite{KYP}.
	
	\begin{lemma}\label{lemma: relations psi_2}
		Under Assumption \ref{assump: stability} we have (i)~$\psi_2(0) = 0$,
		and (ii)~$\varphi_2'(0)\psi_2'(0) = 1$.
	\end{lemma}
	
	From \eqref{eq: matrix equation chi(0) and chi(v)}, we have
	\begin{equation}
		\chi(0,\alpha,\beta)
		=
		\frac{(1-P_{22})b_1+P_{12}b_2}
		{\det(\bs{I}-\bs{P})}.
	\end{equation}
	As $\beta\downarrow0$, both the numerator and denominator converge to zero. Indeed, $b_1$ and $b_2$ are of order $\beta$, while $\lim_{\beta\downarrow0}\bs{P}$ is stochastic and hence
	\[
	\lim_{\beta\downarrow0}\det(\bs{I}-\bs{P})=0.
	\]
	We therefore expand the numerator and denominator around $\beta=0$ and retain only their first-order terms. The resulting calculations are routine and are done in Appendix \ref{app: proof thrm stationary LST}.
	
	\begin{proposition}\label{theorem: stationary LST workload}
		For $\alpha\geqslant0$,
		\begin{equation}\label{eq: stationary LST workload}
			\mathbb{E}e^{-\alpha V_v}
			=
			\frac{f(v,\alpha)}{f(v,0)},
		\end{equation}
		where
		\begin{align*}
			f(v,\alpha)
			&:=
			\frac{1}{\lambda_1}
			+\frac{1-\mathbb{E}(e^{-\alpha B_1})}
			{\varphi_2(\alpha)}
			+\left(
			1-\frac{\varphi_1(\alpha)}
			{\varphi_2(\alpha)}
			\right)
			\int_0^v e^{-\alpha y}W_1^{(0)}(y)
			\left(
			1-
			\int_0^{v-y}e^{-\alpha x}
			\mathbb{P}(B_1\in\mathrm{d}x)
			\right)\mathrm{d}y,
			\\
			f(v,0) &\equiv \lim_{\alpha\downarrow0} f(v,\alpha) 
			=
			\frac{1}{\lambda_1}
			+\psi_2'(0)\,\mathbb{E}(B_1)
			+\left(
			1-\varphi_1'(0)\,\psi_2'(0)
			\right)
			\int_0^v
			W_1^{(0)}(y)
			\mathbb{P}(B_1>v-y)\,\mathrm{d}y.
		\end{align*}
	\end{proposition}
	
	We conclude this subsection by validating the derived stationary workload LST by considering three special cases for which known results are available: (i) $Y_1=Y_2$,
	(ii) $v\to\infty$, and (iii) $v\downarrow0$. The first case provides a
	simple consistency check, while the latter two give non-trivial
	limiting regimes.
	
	\subsubsection*{(i) Equal input processes}
	
	If $Y_1=Y_2$, then
	$\varphi_1(\alpha)=\varphi_2(\alpha)$, and all terms involving their
	difference vanish. Hence \eqref{eq: stationary LST workload} reduces to
	\begin{equation}\label{eq:MG1}
		\mathbb{E}e^{-\alpha V_v}
		=
		\left({\frac{1}{\lambda_1}
			+\frac{1-\mathscr{B}_1(\alpha)}{\varphi_1(\alpha)}}\right)\left/
		\left({\frac{1}{\lambda_1}+\psi_1'(0)\mathbb{E}(B_1)}\right)\right..
	\end{equation}
	Using
	$
	\psi_1'(0)
	=
	({-r_1-\lambda_1\mathbb{E}(B_1)})^{-1},
	$
	this is the standard LST of the stationary workload of an M/G/1 queue
	driven by $Y_1$; see e.g.\ \cite[Theorem 3.1]{DM}.
	
	\subsubsection*{(ii) The limit $v\to\infty$}
	
	We next consider the limit $v\to\infty$. We assume
	$\mathbb{E}(Y_1(1))<0$, so that $\varphi_1'(0)>0$ and
	$\psi_1(0)=0$. The limiting expressions for the two integral terms
	appearing in $f(v,\alpha)$ and $f(v,0)$ are
	\begin{align}
		\lim_{v\to\infty}
		\int_0^v e^{-\alpha x}
		\int_0^{v-x}e^{-\alpha y}W_1^{(0)}(y)\,\mathrm{d}y\,
		\mathbb{P}(B_1\in\mathrm{d}x)
		&=
		\frac{\mathscr{B}_1(\alpha)}{\varphi_1(\alpha)}, \label{eq: limit_integral_I}\\
		\lim_{v\to\infty}
		\int_0^v W_1^{(0)}(y)
		\mathbb{P}(B_1>v-y)\,\mathrm{d}y
		&=
		\frac{\mathbb{E}(B_1)}{\varphi_1'(0)}. \label{eq: limit_integral_II}
	\end{align}
	The first equality is trivial as the domain of integration becomes $\{(x,y) \in \mathbb{R}^2: x,y\geqslant0\}$. The
	second follows from the first and monotone convergence after rewriting the tail probability and
	letting $\alpha\downarrow0$; see Appendix~\ref{app: limit_integrals} for
	the details.
	
	Substituting \eqref{eq: limit_integral_I}--\eqref{eq: limit_integral_II}
	into $f(v,\alpha)$ and $f(v,0)$ gives
	\begin{align}
		\lim_{v\to\infty} f(v,\alpha)
		&=
		\frac{1}{\lambda_1}
		+\frac{1-\mathscr{B}_1(\alpha)}{\varphi_1(\alpha)},\\
		\lim_{v\to\infty} f(v,0)
		&=
		\frac{1}{\lambda_1}
		+\frac{\mathbb{E}(B_1)}{\varphi_1'(0)}
		=
		\frac{1}{\lambda_1}
		+\psi_1'(0)\mathbb{E}(B_1).
	\end{align}
	It now follows immediately that \eqref{eq:MG1} holds. 
	Thus, as $v$ becomes arbitrarily large, the threshold no longer affects
	the stationary workload distribution.
	
	\subsubsection*{(iii) The limit $v\downarrow0$}
	
	Finally, consider the limit $v\downarrow0$. In this regime, the
	threshold disappears and the system only distinguishes between an empty
	and a non-empty state. Since all integral terms in $f(v,\alpha)$ and
	$f(v,0)$ vanish,
	\begin{equation}
		\lim_{v\downarrow0}\mathbb{E}e^{-\alpha V_v}
		=
		\left({\frac{1}{\lambda_1}
			+\frac{1-\mathscr{B}_1(\alpha)}{\varphi_2(\alpha)}}\right)\left/\left(
		{\frac{1}{\lambda_1}
			+\psi_2'(0)\mathbb{E}(B_1)}\right)\right..
		\label{eq: zero_threshold_LST}
	\end{equation}
	In view of 
	$
	\psi_2'(0)
	=
	({-r_2-\lambda_2\mathbb{E}(B_2)})^{-1},$
	we obtain
	\begin{equation}
		\frac{
			\left(\varphi_2(\alpha)
			+\lambda_1(1-\mathscr{B}_1(\alpha))\right)
			(-r_2-\lambda_2\mathbb{E}(B_2))
		}
		{
			\varphi_2(\alpha)
			\left(-r_2-\lambda_2\mathbb{E}(B_2)
			+\lambda_1\mathbb{E}(B_1)\right)
		}.
		\label{eq: zero_threshold_LST_simplified}
	\end{equation}
	
	To compare this with the model studied in \cite{WEL}, set
	$\lambda_1=\lambda_2\equiv\lambda$ and $r_2=-1$. Then
	\eqref{eq: zero_threshold_LST_simplified} becomes
	\begin{equation}
		\frac{
			\left(\alpha-\lambda
			(\mathscr{B}_1(\alpha)-\mathscr{B}_2(\alpha))\right)
			(1-\lambda\mathbb{E}(B_2))
		}
		{
			\left(\alpha-\lambda(1-\mathscr{B}_2(\alpha))\right)
			\left(1-\lambda(\mathbb{E}(B_2)-\mathbb{E}(B_1))\right)
		}.
	\end{equation}
	This coincides with the LST of the stationary waiting time in
	\cite[Theorem 4]{WEL}. Hence, the derived expression also recovers the
	known zero-threshold model as a special case.

	\begin{remark}{\em 
			In the special case
			\[
			\varphi_i(\alpha)
			= g_i\alpha-\lambda_i\frac{\alpha}{\mu_i+\alpha},
			\qquad i=1,2,
			\]
			the model corresponds to an M/M/1 queue with arrival rate
			$\lambda_i$, service rate $\mu_i$, and service speed $r_i=-g_i$ in
			mode $i$. In this case, Proposition~\ref{theorem: stationary LST workload}
			gives
			\[
			f(v,\alpha)
			=
			\frac{1}{\lambda_1}
			+\frac{1}{g_1\alpha+g_1\mu_1-\lambda_1}
			+e^{-(\alpha+\mu_1-\lambda_1/g_1)v}
			\left[
			\frac{\mu_2+\alpha}{\mu_1+\alpha}
			\frac{1}{g_2\alpha+g_2\mu_2-\lambda_2}
			-\frac{1}{g_1\alpha+g_1\mu_1-\lambda_1}
			\right],
			\]
			with $f(v,0)$ given by the same expression at $\alpha=0$.
			Inverting the LST shows that $V_v$ has an atom of size
			$1/(\lambda_1 f(v,0))$ at zero and density, with $\nu_i:=\mu_i-\lambda_i/g_i$ for $i=1,2$,
			\[
			\frac{1}{f(v,0)}
			\left[
			\frac{1}{g_1}e^{-\nu_1 x}
			+e^{-\nu_1 v}
			\left(
			U_1e^{-\mu_1(x-v)}
			+\frac{U_2}{g_2}e^{-\nu_2(x-v)}
			-\frac{1}{g_1}e^{-\nu_1(x-v)}
			\right)\mathbf{1}_{\{x>v\}}
			\right],
			\]
			where
			\[
			U_1:=\frac{\mu_2-\mu_1}{g_2(\mu_2-\mu_1)-\lambda_2},
			\qquad
			U_2:=-\frac{\lambda_2}{g_2(\mu_2-\mu_1)-\lambda_2}.
			\]
			For $x>v$, the first and last terms in this density cancel. \hfill$\Diamond$}
	\end{remark}

	\subsection{Numerical example}
	We conclude this section by presenting numerical results. We consider an instance of the
	queueing model and evaluate the mean workload, the variance of the workload, and the
	probability that the system is empty. To obtain these performance measures at
	\textit{deterministic} time points, we use the numerical Laplace inversion algorithm
	from~\cite[\S8]{AW2}. The required transforms of the mean, variance, and empty-system
	probability at an exponential time $T_\beta$ are derived from
	Theorem~\ref{theorem: transient LST chi(x)}.
	
	For the instance considered, we work with the same CPPs $Y_1$ and $Y_2$ as in Subsection \ref{sec:numruin}. The threshold lies at $v=4$.
	
	In Figure \ref{fig: queueing model example} we present the mean workload, its variance, and the
	probability of an empty system as functions of time, for different initial workload values, as well as the LST of the steady-state workload.
	The figure also includes Monte Carlo estimates with 95\% confidence intervals for comparison, obtained from $5\cdot 10^3$ samples for the mean, variance and probability of an empty system, and $10^3$ samples for the estimates of the LST. The solid lines represent the values obtained by numerical Laplace inversion,
	while the markers show the corresponding Monte Carlo estimates. 
	
	\begin{figure}[ht]
		\centering
		\begin{subfigure}{0.49\textwidth}
			\includegraphics[width=\textwidth]{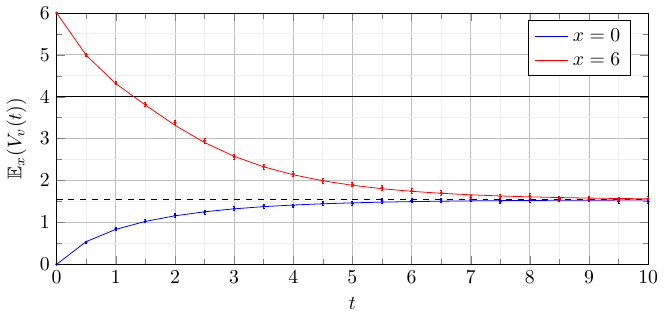}
			\caption{\label{fig: Queueing model mean}}
		\end{subfigure}
		\hfill
		\begin{subfigure}{0.49\textwidth}
			\includegraphics[width=\textwidth]{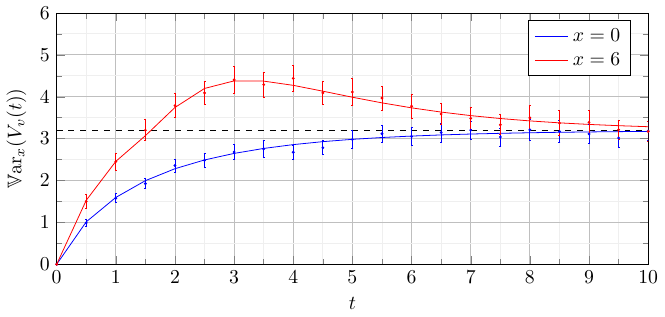}
			\caption{\label{fig: Queueing model var}}
		\end{subfigure}
		\hfill
		\begin{subfigure}{0.49\textwidth}
			\includegraphics[width=\textwidth]{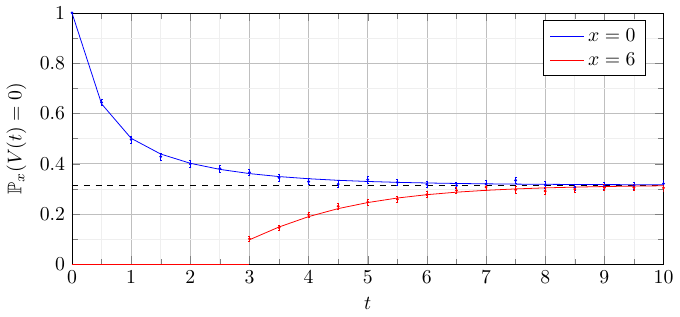}
			\caption{\label{fig: Queueing model empty}}
		\end{subfigure}
		\hfill
		\begin{subfigure}{0.49\textwidth}
			\includegraphics[width=\textwidth]{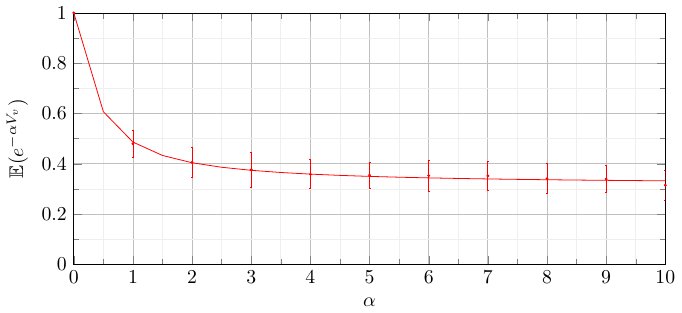}
			\caption{\label{fig: Queueing model SS dens}}
		\end{subfigure}
		\caption{\label{fig: queueing model example}Mean ({\sc a}), variance ({\sc b}) and probability of an empty system ({\sc c}) as functions of time for the instance of the queueing model, and the LST of the steady-state ({\sc d}) as function of $\alpha$. The dashed lines denote the stationary values of the system, which are approximately $1.54$, $3.19$ and $0.32$ respectively.}
	\end{figure}

	\section{{Queueing model: Capacity constraints and the running maximum}}\label{sec: expansions}
	In this section, we consider two further aspects of the queueing model introduced in Section~\ref{sec: queueing model}: (i) an extension incorporating a capacity above the threshold $v$, and (ii) the running maximum of the queue. Further extensions are discussed in Section~\ref{sec:disc}.

	\subsection{Finite-capacity {queueing model}}\label{sec: capacity}
	
	As a first extension, we impose a capacity $K>v$ on the workload system. Let $(\hat{V}_v(t))_{t\geqslant0}$ denote the resulting workload process, with $0\leqslant\hat{V}_v(t)\leqslant K$. We assume partial rejection: if an arriving job causes the workload to exceed $K$, the workload is set to $K$ and the excess work is discarded. For an initial workload $x\in[0,K]$, we study the LST
	\begin{equation}
		\hat\chi(x)\equiv\hat\chi(x,\alpha,\beta)
		:=\mathbb{E}_x\left(e^{-\alpha\hat{V}_v(T_\beta)}\right)
		=\mathbb{E}\left(e^{-\alpha\hat{V}_v(T_\beta)}
		\,\middle|\,\hat{V}_v(0)=x\right).
	\end{equation}

	\subsubsection*{Decomposition}
	In the following proposition, which is similar to Proposition \ref{prop: chi(x)}, we show how $\hat\chi(x)$ can be expressed in terms of $\hat\chi(0), \hat\chi(v)$ and $\hat\chi(K)$. 
	\begin{proposition}\label{prop: chi(x) capacity}
		Let $\beta>0$.  For $x\in(v,K]$
		\begin{equation}\label{eq: chi(x) capacity v < x < K}
			\hat\chi(x) = \delta_{2,-}(x-v, K-x, \beta)\hat\chi(v) + \delta_{2,+}(x-v, K-x, \beta)\hat\chi(K) + \delta_{2,\star}(x-v, K-x, \beta),
		\end{equation}
		and for  $x\in[0,v]$
		\begin{align}\label{eq: chi(x) capacity 0 < x < v}
			\notag\hat\chi(x) =&\: \delta_{1,-}(x, v-x, \beta)\hat\chi(0) + \int_0^{K-v} \tilde{h}_1(x,v-x,y,\beta)\delta_{2,-}(y,(K-v)-y,\beta)\,\mathrm{d}y\,\hat\chi(v) \:+\\
			\notag&\: \int_0^{K-v} \tilde{h}_1(x,v-x,y,\beta)\delta_{2,+}(y,(K-v)-y,\beta)\,\mathrm{d}y\,\hat\chi(K) \:+\\
			\notag&\: \int_{K-v}^\infty \tilde{h}_1(x,v-x,y,\beta)\,\mathrm{d}y\,\hat\chi(K) \:+\\
			&\: \int_0^{K-v} \tilde{h}_1(x,v-x,y,\beta)\delta_{2,\star}(y,(K-v)-y,\beta)\,\mathrm{d}y + \delta_{1,\star}(x, v-x, \beta).
		\end{align}
	\end{proposition}
	\begin{proof}
		We start with the first claim. For $v < x \leqslant K$, conditioning on the first event, which is either killing, hitting $v$, or hitting $K$, and applying the strong Markov property yields \eqref{eq: chi(x) capacity v < x < K}.
		
		For the second claim, for $0 < x \leqslant v$, conditioning on the first event and applying the strong Markov property gives
		\begin{align*}
			\notag\hat\chi(x) &= \delta_{1,-}(x, v-x, \beta)\hat\chi(0) + \int_0^{K-v} \tilde{h}_1(x,v-x,y,\beta)\hat\chi(v+y)\,\mathrm{d}y\\
			&\:\:\:\:\:+ \int_{K-v}^\infty \tilde{h}_1(x,v-x,y,\beta)\,\mathrm{d}y\,\hat\chi(K) + \delta_{1,\star}(x, v-x, \beta).
		\end{align*}
		Substituting \eqref{eq: chi(x) capacity v < x < K} into the first integral yields
		\begin{align*}
			\int_0^{K-v}  \tilde{h}_1(x,v-x,y,\beta)\hat\chi(v+y)\,\mathrm{d}y 
			=\:& \int_0^{K-v} \tilde{h}_1(x,v-x,y,\beta)\delta_{2,-}(y,(K-v)-y,\beta)\,\mathrm{d}y\,\hat\chi(v) \:+\\
			& \int_0^{K-v} \tilde{h}_1(x,v-x,y,\beta)\delta_{2,+}(y,(K-v)-y,\beta)\,\mathrm{d}y\,\hat\chi(K)\:+ \\
			& \int_0^{K-v} \tilde{h}_1(x,v-x,y,\beta)\delta_{2,\star}(y,(K-v)-y,\beta)\,\mathrm{d}y,
		\end{align*}
		which yields \eqref{eq: chi(x) capacity 0 < x < v}.
	\end{proof}
	
	\subsubsection*{Derivation of \texorpdfstring{$\hat\chi(0)$}{chi 0}, \texorpdfstring{$\hat\chi(v)$}{chi v} and \texorpdfstring{$\hat\chi(K)$}{chi K}}
	We characterize  $\hat{\bs \chi}:=(\hat\chi(0) ,
	\hat\chi(v) ,
	\hat\chi(K))^\top$ by a system of three linear equations. We first decompose $\hat\chi(0)$ in terms of these quantities, while $\hat\chi(v)$ and $\hat\chi(K)$ follow from the decomposition in Proposition \ref{prop: chi(x) capacity}. We define the $3\times3$ matrix $\hat{\bs{P}}$ through its entries: for the first row, we have
	\begin{align*}
		\hat{P}_{11} &:= \frac{\lambda_1}{\lambda_1 + \beta}\int_{[0,v]}\mathbb{P}(B_1\in\mathrm{d}x)\delta_{1,-}(x,v-x,\beta),\\
		\hat{P}_{12} &:= \frac{\lambda_1}{\lambda_1+\beta}\int_{[0,v]} \mathbb{P}(B_1\in\mathrm{d}x)\int_{0}^{K-v}\tilde{h}_1(x,v-x,y,\beta)\delta_{2,-}(y,K-v-y,\beta)\,\mathrm{d}y \\
		&\:\:\:\:\:\:\:\:\:\: + \frac{\lambda_1}{\lambda_1 + \beta}\int_{(v,K]} \mathbb{P}(B_1\in\mathrm{d}x)\delta_{2,-}(x-v,K-x,\beta), \\
		\hat{P}_{13} &:= \frac{\lambda_1}{\lambda_1+\beta}\int_{[0,v]} \mathbb{P}(B_1\in\mathrm{d}x)\int_{0}^{K-v}\tilde{h}_1(x,v-x,y,\beta)\delta_{2,+}(y,K-v-y,\beta)\,\mathrm{d}y \\
		&\:\:\:\:\:\:\:\:\:\: + \frac{\lambda_1}{\lambda_1 + \beta}\int_{(v,K]} \mathbb{P}(B_1\in\mathrm{d}x)\delta_{2,+}(x-v,K-x,\beta) + \frac{\lambda_1}{\lambda_1 + \beta} \int_{(K,\infty)} \mathbb{P}(B_1\in\mathrm{d}x),
	\end{align*}
	the second row is given by
	\begin{align*}
		\hat{P}_{21} &:= \delta_{1,-}(v,0,\beta),\qquad
		\hat{P}_{22} := \int_0^{K-v} \tilde{h}_1(v,0,y,\beta)\delta_{2,-}(y,K-v-y,\beta)\,\mathrm{d}y, \\
		\hat{P}_{23} &:= \int_0^{K-v} \tilde{h}_1(v,0,y,\beta)\delta_{2,+}(y,K-v-y,\beta)\,\mathrm{d}y + \int_{K-v}^\infty \tilde{h}_1(v,0,y,\beta)\,\mathrm{d}y,
	\end{align*}
	and finally, for the third row we have
	\begin{align*}
		\hat{P}_{31} &:= 0,\qquad
		\hat{P}_{32} := \delta_{2,-}(K-v,0,\beta), \qquad
		\hat{P}_{33} := \delta_{2,+}(K-v,0,\beta).
	\end{align*}
	We also define a vector $\hat{\bs{b}}\in\mathbb{R}^3$ through its entries:
	\begin{align*}
		\hat{b}_1 &:= \frac{\beta}{\lambda_1 + \beta} + \frac{\lambda_1}{\lambda_1 + \beta}\int_{[0,v]} \mathbb{P}(B_1\in\mathrm{d}x)\delta_{1,\star}(x,v-x,\beta) \\
		&\:\:\:\:\:\:\:\:\:\: + \frac{\lambda_1}{\lambda_1 + \beta}\int_{(v,K]} \mathbb{P}(B_1\in\mathrm{d}x)\delta_{2,\star}(x-v,K-x,\beta), \\
		\hat{b}_2 &:= \int_0^{K-v} \tilde{h}_1(v,0,y,\beta)\delta_{2,\star}(y,(K-v)-y,\beta)\,\mathrm{d}y + \delta_{1,\star}(v, 0, \beta), \qquad
		\hat{b}_3 := \delta_{2,\star}(K-v, 0, \beta).
	\end{align*}

	\begin{lemma}\label{lemma: expansion2 chi(0) chi(v) chi(K)}
		For $\beta>0$, 
		\begin{equation}
			\hat{\bs \chi}
			= 
			(\bs{I} - \hat{\bs P})^{-1}
			\hat{\bs{b}}.
		\end{equation}
	\end{lemma}
	\begin{proof}
		The second and third entries of $(\bs{I}-\hat{\bs P})\hat{\bs\chi}=\hat{\bs b}$ follow directly from Proposition~\ref{prop: chi(x) capacity}, with $x=v$ and $x=K$, respectively.
		For the first entry, corresponding to a system that is initially empty, we again condition on the first event. In this case this is either killing or an arrival from $Y_1$. In the case of an arrival, the job 
		size can either be less than or equal to $v$, greater than $v$ and less than or equal to $K$, or greater than $K$. This gives
		\begin{align*}
			\hat\chi(0) &= \frac{\beta}{\lambda_1 + \beta} + \frac{\lambda_1}{\lambda_1+\beta}\int_0^\infty \mathbb{P}\left(B_1\in\mathrm{d}x\right)\hat\chi(x),
		\end{align*}
		where
		\begin{align*}
			\int_0^\infty \mathbb{P}\left(B_1\in\mathrm{d}x\right)\hat\chi(x)&= \int_{[0,v]} \mathbb{P}\left(B_1\in\mathrm{d}x\right)\hat\chi(x) + \int_{(v,K]} \mathbb{P}\left(B_1\in\mathrm{d}x\right)\hat\chi(x) + \int_{(K,\infty)} \mathbb{P}\left(B_1\in\mathrm{d}x\right)\hat\chi(K).
		\end{align*}
		Substituting the decompositions for $\chi(x)$ as given in Proposition \ref{prop: chi(x) capacity} under their respective integrals gives the first entry of $(\bs{I} - \hat{\bs P})\hat{\bs \chi}
		= 
		\hat{\bs{b}}$.
		In order to show that $\bs{I} - \hat{\bs P}$ is invertible, it suffices to observe that the row-sums of $\hat{\bs P}$ are less than one entry-wise. 
	\end{proof}

	\subsection{Running maximum of queueing model}\label{sec: running max queueing}
	
	In this subsection we aim to characterize the distribution of the running maximum of the workload process. In our analysis of this specific process we rely on the methods used in Sections \ref{sec: Cramer-Lundberg model}--\ref{sec: queueing model}. Related results on maxima of workload processes over random time intervals can be
	found in the literature, although the focus is often on maxima over busy periods or
	regenerative cycles; see, e.g., \cite{AP,ASM3}. For independently
	sampled time intervals, \cite{COH} derives, in particular, the LST of the maximum workload over an exponentially distributed interval in the
	M/G/1 queue; see \cite[Eqns.~(5.76)--(5.79)]{COH}. An alternative approach has been developed in \cite[Exercise 5.4]{MB}.
	
	Let $\bar{V}_v(t):=\sup_{s\in[0,t]}V_v(s)$ denote the running maximum of the workload process. We analyze its tail probability at an independent time $T_\beta\sim\mathrm{Exp}(\beta)$, $\beta>0$,
	\begin{equation*}
		p^V(u,v,x,\beta):=\mathbb{P}\bigl(\bar{V}_v(T_\beta)>u \big|V_v(0)=x\bigr), \qquad u\geqslant0.
	\end{equation*}
	As in Section \ref{sec: Cramer-Lundberg model}, we distinguish cases according to $x$ relative to $u$, with $p^V(u,v,x,\beta)=1$ for $u<x$. We include the case $0\leqslant u\leqslant v$, where the probability under study reduces to the running maximum of the M/G/1 queue,  for later reference.

	\begin{lemma}
		Let $0\leqslant u \leqslant v$ and $\beta > 0$.  For $x=0$
		\begin{align}\label{eq: p^V for 0<=u<=v and x=0}
			p^V(u,v,0,\beta) = \frac{\int_{(u,\infty)} \mathbb{P}(B_1\in\mathrm{d}y) + \int_{[0,u]} \mathbb{P}(B_1\in\mathrm{d}y)\,\delta_{1,+}(y,u-y,\beta)}{\displaystyle{{\frac{\lambda_1 + \beta}{\lambda_1}}} - \int_{[0,u]} \mathbb{P}(B_1\in\mathrm{d}y)\,\delta_{1,-}(y,u-y,\beta)},
		\end{align}
		and for $0<x\leqslant u$
		\begin{align}\label{eq: p^V for 0<=u<=v and 0<x<=v}
			p^V(u,v,x,\beta) = \delta_{1,+}(x,u-x,\beta) + \delta_{1,-}(x,u-x,\beta)p^V(u,v,0,\beta).
		\end{align}
	\end{lemma}
	\begin{proof}
		For $x=0$, conditioning on the first arrival and on whether killing occurs before it gives
		\begin{align*}
			p^V(u,v,0,\beta) &= \frac{\lambda_1}{\lambda_1 + 
				\beta}\Bigg(\int_{(u,\infty)}\mathbb{P}(B_1\in\mathrm{d}y) + \int_{[0,u]}\mathbb{P}(B_1 \in \mathrm{d}y)\,\delta_{1,+}(y,u-y,\beta) \\
			&\:\:\:\:\:\:\:\:\:\: + \int_{[0,u]}\mathbb{P}(B_1 \in \mathrm{d}y)\,\delta_{1,-}(y,u-y,\beta)\,p^V(u,v,0,\beta)\Bigg).
		\end{align*}
		Rearranging yields \eqref{eq: p^V for 0<=u<=v and x=0}.
		For $0<x\leqslant u$, either $u$ is crossed before the workload reaches zero or the system empties first, in which case the process restarts from $0$. This immediately yields \eqref{eq: p^V for 0<=u<=v and 0<x<=v}.
	\end{proof}

	\begin{lemma}
		Let $ u > v$ and $\beta > 0$. For $x\in (u,v)$
		\begin{align}\label{eq: p^V for u>v and v<x<u}
			p^V(u,v,x,\beta) = \delta_{2,+}(x-v, u-x, \beta) + \delta_{2,-}(x-v, u-x, \beta)p^V(u,v,v,\beta),
		\end{align}
		and for $x\in(0, v]$
		\begin{align}
			\notag p^V(u,v,x,\beta) &= \int_{u-v}^\infty \tilde{h}_1(x,v-x,y,\beta)\,\mathrm{d}y + \int_0^{u-v} \tilde{h}_1(x,v-x,y,\beta)\delta_{2,+}(y,u-v-y,\beta)\,\mathrm{d}y \\
			\notag&\:\:\:\:\:+\int_0^{u-v} \tilde{h}_1(x,v-x,y,\beta)\delta_{2,-}(y,u-v-y,\beta)\,\mathrm{d}y\cdot p^V(u,v,v,\beta) \\
			&\:\:\:\:\:+\delta_{1,-}(x,v-x,\beta)p^V(u,v,0,\beta). \label{eq: p^V for u>v and 0<x<=v}
		\end{align}
	\end{lemma}
	\begin{proof}
		In the decomposition of the ruin probabilities, we use the hitting times of the levels
		$0$ and $v$ as regenerative points.
		
		For the first claim, condition on whether the workload process crosses $u$ before reaching $v$
		and before killing, or reaches $v$ first. This directly gives the two terms on the
		right-hand side of \eqref{eq: p^V for u>v and v<x<u}: the first corresponds to
		crossing $u$ before reaching $v$ or being killed, while the second corresponds to
		reaching $v$ first and subsequently crossing $u$ starting from $v$.
		
		For the second claim, starting from $0<x\leqslant v$, we distinguish the possible regenerative events.
		The first term on the right-hand side of
		\eqref{eq: p^V for u>v and 0<x<=v} corresponds to crossing $v$ with an overshoot
		large enough to cross $u$ as well. The second corresponds to crossing $v$ with a
		smaller overshoot and subsequently crossing $u$ before returning to $v$ or being
		killed. The third term covers the case where $v$ is reached again before $u$ is
		crossed or the process is killed. Finally, the fourth term corresponds to reaching
		$0$ before crossing $v$ or being killed.
	\end{proof}
	
	\subsubsection*{Derivation of \texorpdfstring{$p^V(u,v,0,\beta)$}{} and \texorpdfstring{$p^V(u,v,v,\beta)$}{}}
	Similar to the derivation of the main queueing model and the extension in Subsection \ref{sec: capacity},
	$p^V(u,v,v,\beta)$ is obtained from \eqref{eq: p^V for u>v and 0<x<=v} by
	separating the terms containing $p^V(u,v,v,\beta)$, yielding
	\begin{align*}
		p^V(&u,v,v,\beta) = \left({1 - \int_0^{u-v} \tilde{h}_1(v,0,y,\beta)\delta_{2,-}(y,u-v-y,\beta)\,\mathrm{d}y}\right)^{-1}\cdot \\
		&\Bigg( \int_{u-v}^\infty \tilde{h}_1(v,0,y,\beta)\,\mathrm{d}y + \int_0^{u-v} \tilde{h}_1(v,0,y,\beta)\delta_{2,+}(y,u-v-y,\beta)\,\mathrm{d}y+\delta_{1,-}(v,0,\beta)p^V(u,v,0,\beta) \Bigg).
	\end{align*}
	Likewise, conditioning on the first transition when the system is initially empty
	gives
	\begin{align*}
		p^V(u,v,0,\beta) &= \frac{\lambda_1}{\lambda_1 + \beta}\Bigg( \int_u^\infty \mathbb{P}(B_1\in\mathrm{d}y) + \int_v^u \mathbb{P}(B_1\in\mathrm{d}y)\delta_{2,+}(y-v,u-y,\beta) \\
		\:\:\:\:\:&+ \int_v^u \mathbb{P}(B_1\in\mathrm{d}y)\delta_{2,-}(y-v,u-y,\beta)\cdot p^V(u,v,v,\beta) \\
		\:\:\:\:\:&+ \int_0^v \mathbb{P}(B_1\in\mathrm{d}y)\delta_{1,-}(y,v-y,\beta)\cdot p^V(u,v,0,\beta) \\
		\:\:\:\:\:&+ \int_0^v \mathbb{P}(B_1\in\mathrm{d}y)\int_{u-v}^\infty\tilde{h}_{1}(y,v-y,z,\beta)\,\mathrm{d}z \\ 
		\:\:\:\:\:&+ \int_0^v \mathbb{P}(B_1\in\mathrm{d}y)\int_0^{u-v}\tilde{h}_{1}(y,v-y,z,\beta)\delta_{2,+}(z, u-v-z,\beta)\,\mathrm{d}z \\
		\:\:\:\:\:&+ \int_0^v \mathbb{P}(B_1\in\mathrm{d}y)\int_0^{u-v}\tilde{h}_{1}(y,v-y,z,\beta)\delta_{2,-}(z, u-v-z,\beta)\,\mathrm{d}z \cdot p^V(u,v,v,\beta)\Bigg).
	\end{align*}
	Therefore, the two decompositions together form a system of two linear equations in
	$p^V(u,v,0,\beta)$ and $p^V(u,v,v,\beta)$. Solving this system yields the desired
	ruin probabilities.

	\section{Concluding Remarks}\label{sec:disc}
	In this paper, we analyzed the transient behavior of a threshold-dependent Cram\'er--Lundberg model with an {arbitrary} initial surplus and an M/G/1-type queueing model with an arbitrary initial workload. In both models, the dynamics are driven by different compound Poisson processes with drift depending on whether the process is below or above a fixed threshold. We developed procedures to characterize, for the Cram\'er--Lundberg model, the ruin probability before an exponentially distributed epoch and, for the queueing model, the LST of the workload at an exponentially distributed epoch. The resulting Laplace transforms can be used to determine the density of the transient workload via Laplace inversion techniques. For phase-type jump-size distributions, the auxiliary quantities appearing in these characterizations admit explicit representations. 
	
	\medskip
	
	The fluctuation-theoretic results used in our analysis also reveal several interesting connections and possible directions for future research. In particular, a natural extension of the present model would be to consider the broader class of spectrally positive L\'evy processes, although extra care is required in its definition due to possible fluctuations around the threshold level. Our approach, however, relies on properties specific to compound Poisson processes. These properties are used, in particular, in the derivation of the ruin probability with threshold $v=0$ (see Lemma~\ref{lemma: ruin prob v=0}) and in the derivation of the LSTs of the workload initialized at $0$ or at $v$, namely $\chi(0)$ and $\chi(v)$, in Subsection~\ref{sec: chi 0 and chi K}. For a general spectrally positive L\'evy process, these arguments need not carry over. For instance, when the process has a positive Brownian coefficient, both the decomposition of the ruin probability in terms of itself, used in the proof of Lemma~\ref{lemma: ruin prob v=0}, and the decomposition of $\chi(v)$ in terms of $\chi(v)$ and $\chi(0)$, obtained from Proposition~\ref{prop: chi(x)} with $x=v$, reduce to identities that provide no information from which the desired quantities can be determined. Moreover, the derivation of $\chi(0)$ relies on conditioning on the first event, which, in the compound Poisson-driven queueing systems considered here, is either an arrival or a killing event. Extending the analysis to general spectrally positive L\'evy processes would therefore require new arguments and, in particular, new results for these quantities.
	
	\medskip
	
	There are several other directions in which the models could be extended. Multiple thresholds provide a natural generalization, although the resulting derivations would become substantially more involved. For the queueing model, one could also study other transient quantities, such as the cumulative time spent idle or above the threshold; cf.\ \cite[Section 5]{MRS} for an approach that could be followed. The approach in Subsection \ref{sec: running max queueing} may likewise be used to investigate the tail probability of the workload at an exponentially distributed epoch, rather than the tail probability of its running maximum. Finally, given the dynamics of the compound Poisson processes, the threshold itself could be optimized with respect to a specified performance measure or cost function; cf.\ the framework of \cite[Section 6]{ABKM}. While exact solutions to such optimization problems may be difficult to obtain, the characterizations developed in this paper could provide a basis for their numerical investigation.

	\appendix
	
	\section{Equivalence proofs}
	
	The scale-function and martingale methods used in Section~\ref{sec: aux for phase-type} led to different characterizations of $\eta(u_+,\alpha,\beta)$, $\tilde\eta(u_-,u_+,\alpha,\beta)$, and the hitting probabilities. It is therefore natural to ask whether these characterizations are equivalent. In this appendix, we establish the equivalence of the two approaches by showing that the resulting expressions coincide.
	
	We first introduce some notation. Recall that
	$\mathscr{B}(\alpha)=\bs{a}^\top(\alpha\bs{I}-\bs{S})^{-1}\bs{s}$.
	Let
	\[
	p(\alpha):=\det(\alpha\bs{I}-\bs{S}),
	\qquad
	q_j(\alpha):=\det(\alpha\bs{I}-\bs{S})_{j,\bs{s}},
	\quad j\in[d],
	\]
	where $(\alpha\bs{I}-\bs{S})_{j,\bs{s}}$ is obtained by replacing the
	$j$th column of $\alpha\bs{I}-\bs{S}$ by $\bs{s}$. Thus, $p$ is a monic
	polynomial of degree $d$, while each $q_j$ has degree at most $d-1$.
	Cramer's rule gives
	\begin{equation}\label{eq: LSTs phase-type in polynomials}
		\mathscr{B}(\alpha)
		= \frac{\bs{a}^\top\bs{q}(\alpha)}{p(\alpha)},
		\qquad
		\mathscr{B}_j(\alpha)
		= \frac{q_j(\alpha)}{p(\alpha)}.
	\end{equation}
	
	Similarly,
	\begin{equation*}
		\varphi(\alpha)-\beta
		= -r\alpha-\lambda(1-\mathscr{B}(\alpha))-\beta
		= \frac{\tilde q(\alpha)}{p(\alpha)},
	\end{equation*}
	where
	\begin{equation}\label{eq: def tilde q}
		\tilde q(\alpha)
		:=(-r\alpha-\lambda-\beta)p(\alpha)
		+\lambda\sum_{j=1}^d a_jq_j(\alpha).
	\end{equation}
	Since $\tilde q$ has degree $d+1$ with leading coefficient $-r$, and
	$\alpha_1,\ldots,\alpha_{d+1}$ are the roots of the CL-equation,
	\begin{equation}\label{eq: tilde q as prod of roots}
		\tilde q(\alpha)
		=(-r)\prod_{k=1}^{d+1}(\alpha-\alpha_k).
	\end{equation}
	Hence,
	\begin{equation}\label{eq: partial frac decomp p / tilde q}
		\frac{1}{\varphi(\alpha)-\beta}
		=\frac{p(\alpha)}{\tilde q(\alpha)}
		=\sum_{k=1}^{d+1}\frac{c_k}{\alpha-\alpha_k},
	\end{equation}
	with
	\begin{equation}\label{eq: def c_k}
		c_k
		:=\frac{p(\alpha_k)}
		{(-r)\prod_{\ell\neq k}^{d+1}(\alpha_k-\alpha_\ell)}.
	\end{equation}
	
	We first prove the invertibility of $\bs{B}$ and then establish the
	equivalence of the two methods.
	\subsection{Invertibility of \texorpdfstring{$\bs{B}$}{}}\label{app: invertibility matrices}
	Recall that $\bs{B}$ is defined as
	\begin{equation*}
		\bs{B}_{kj} = \mathscr{B}_j(\alpha_k) = \frac{q_j(\alpha_k)}{p(\alpha_k)}. 
	\end{equation*} 
	\begin{lemma}\label{lemma:B invertible}
		If $\alpha_1,\ldots,\alpha_d$ are distinct and none is an eigenvalue of
		$\bs S$, then $\bs B$ is nonsingular.
	\end{lemma}
	\begin{proof}
		Suppose that $\bs{Bt}=0$ for some $\bs{t}\in\mathbb{C}^d$. This implies that
		\begin{equation*}
			0 = \sum_{j=1}^d t_j\mathscr{B}_j(\alpha_k) = \sum_{j=1}^d t_j\frac{q_j(\alpha_k)}{p(\alpha_k)} = \frac{\bs{t}^{\top}\bs{q}(\alpha_k)}{p(\alpha_k)}.
		\end{equation*}
		Because none of the $\alpha_k$ is an eigenvalue of $\bs{S}$, $p(\alpha_k) \neq 0$. Hence $\bs{t}^{\top}\bs{q}(\alpha_k) = 0$ for all $k\in[d]$. 
		
		Recall that each $q_j(\alpha)$ is at most of degree $d-1$, hence $\bs{t}^\top\bs{q}(\alpha)$ is at most of degree $d-1$. Since the $\alpha_k$ are distinct we thus have $\bs{t}^\top\bs{q}(\alpha)\equiv0$.
		
		It follows that $(\bs{t}^\top\bs{q}(\alpha)) / p(\alpha) = 0$
		for all $\alpha$ that are not an eigenvalue of $\bs{S}$. Since
		\[
		\frac{\bs{q}(\alpha)}{p(\alpha)} = \left(\alpha\bs{I}-\bs{S}\right)^{-1}\bs{s}
		= \sum_{m=0}^{\infty}\frac{\bs{S}^m\bs{s}}{\alpha^{m+1}}
		\]
		for sufficiently large $\alpha$, we obtain
		\[
		\bs{t}^\top\bs{S}^m\bs{s}=0
		\quad \text{for all} \quad m\geqslant0.
		\]
		For a minimal phase-type representation, the vectors
		$\bs{s},\bs{S}\bs{s},\ldots,\bs{S}^{d-1}\bs{s}$ span $\mathbb{R}^d$, see e.g. \cite[Theorem 4.2.16]{BN}, so we must have $\bs{t}=0$. Thus $\bs{B}$ is nonsingular.
	\end{proof}
	
	\subsection{Verification that \texorpdfstring{\eqref{eq: eta scale functions}}{} and \texorpdfstring{\eqref{eq: eta martingales}}{} are equal}\label{app: equal etas}
	
	We show that \eqref{eq: eta scale functions} and
	\eqref{eq: eta martingales} yield the same rational function of $\alpha$.
	Using \eqref{eq: tilde q as prod of roots}, \eqref{eq: eta scale functions}
	becomes
	\begin{equation*}
		\eta(u_+,\alpha,\beta)
		= \frac{1}{p(\alpha)}
		\sum_{k=1}^{d} c_k
		\left(
		\frac{\tilde q(\alpha)}{\alpha-\alpha_k}
		-\frac{\tilde q(\alpha)}{\alpha-\alpha_{d+1}}
		\right)e^{\alpha_k u_+},
	\end{equation*}
	and hence
	\begin{equation}\label{eq: eta Laplace inverse in theta}
		\eta(u_+,\alpha,\beta)
		= \frac{\bs{\mathfrak C}(\alpha)^\top
			\bs{\theta}(u_+,\beta)}{p(\alpha)},
	\end{equation}
	where
	\begin{align*}
		\mathfrak C_k(\alpha)
		&:=c_k\left(
		\frac{\tilde q(\alpha)}{\alpha-\alpha_k}
		-\frac{\tilde q(\alpha)}{\alpha-\alpha_{d+1}}
		\right)=-rc_k(\alpha_k-\alpha_{d+1})
		\prod_{\substack{\ell=1\\\ell\neq k}}^d
		(\alpha-\alpha_\ell).
	\end{align*}
	
	On the other hand, by the martingale representation
	\eqref{eq: eta martingales} and the invertibility of $\bs B$
	(Lemma~\ref{lemma:B invertible}),
	\begin{equation}\label{eq: eta martingale in theta}
		\eta(u_+,\alpha,\beta)
		=\bs{\mathfrak p}(u_+,\beta)^\top
		(\alpha\bs I-\bs S)^{-1}\bs s
		=\frac{\bs{\mathfrak p}(u_+,\beta)^\top\bs q(\alpha)}
		{p(\alpha)}
		=\frac{\bs q(\alpha)^\top\bs B^{-1}
			\bs\theta(u_+,\beta)}{p(\alpha)}.
	\end{equation}
	The two representations \eqref{eq: eta Laplace inverse in theta} and
	\eqref{eq: eta martingale in theta} therefore have the same denominator
	$p(\alpha)$. Hence, to show that they coincide, it suffices to establish the following identity.
	
	\begin{lemma}\label{lemma: q vs Bc}
		For $\beta > 0$, 
		\begin{equation}\label{eq: q vs Bc}
			\bs{q}(\alpha) = \bs{B}^\top\bs{\mathfrak{C}}(\alpha).
		\end{equation}
	\end{lemma}
	
	\begin{proof}
		We will show that the entries of the two sides of
		\eqref{eq: q vs Bc} are the same polynomials. To this end, let
		$j\in[d]$ be given. For the corresponding entry on the right-hand
		side of \eqref{eq: q vs Bc} we have 
		\begin{equation}\label{eq: q_j = Bc}
			\sum_{k=1}^{d}\bs{B}_{jk}^\top\mathfrak{C}_k(\alpha) = \sum_{k=1}^{d} \frac{q_j(\alpha_k)}{p(\alpha_k)} \cdot (-r)c_k (\alpha_k-\alpha_{d+1}) \prod_{\ell=1, \ell\neq k}^{d} (\alpha - \alpha_\ell).
		\end{equation}
		First note that \eqref{eq: q_j = Bc} is a polynomial in $\alpha$ of
		degree at most $d-1$. We will show that this polynomial agrees with
		$q_j(\alpha)$ at the $d$ points $\alpha_1,\ldots,\alpha_d$. Let
		$n\in[d]$. Evaluating \eqref{eq: q_j = Bc} at $\alpha=\alpha_n$ gives
		\begin{equation}\label{eq: rho_n in}
			\sum_{k=1}^{d} \frac{q_j(\alpha_k)}{p(\alpha_k)} \cdot (-r)c_k (\alpha_k-\alpha_{d+1}) \prod_{\ell=1, \ell\neq k}^{d} (\alpha_n - \alpha_\ell).
		\end{equation}
		For every $k\in[d]$, except $k=n$, the term within the summation is
		zero, so \eqref{eq: rho_n in} reduces to
		\begin{equation}\label{eq: rho_n in 2}
			\frac{q_j(\alpha_n)}{p(\alpha_n)} \cdot (-r)c_n
			\prod_{\ell=1, \ell\neq n}^{d+1} (\alpha_n - \alpha_\ell)
			= q_j(\alpha_n),
		\end{equation}
		by \eqref{eq: def c_k}. Hence the two polynomials agree at
		$\alpha_1,\ldots,\alpha_d$ and are therefore identical. Since this
		holds for all $j\in[d]$, \eqref{eq: q vs Bc} follows.
	\end{proof}

	\subsection{Verification that \texorpdfstring{\eqref{eq: eta tilde scale functions}}{} and \texorpdfstring{\eqref{eq: eta tilde martingales}}{} are equal}\label{app: equal etas tilde}We now verify that the scale-function and martingale representations of
	$\tilde\eta(u_-,u_+,\alpha,\beta)$ coincide.

	Using \eqref{eq: eta Laplace inverse in theta} and the decomposition from
	Lemma~\ref{lemma: eta_tilde}, the expression obtained from the
	scale-function method can be written as
	\begin{equation*}
		\tilde\eta(u_-, u_+, \alpha, \beta)
		= \frac{1}{p(\alpha)}\bs{\mathfrak{C}}(\alpha)^\top
		\left(\bs{\theta}(u_+)
		- \delta_-(u_-,u_+,\beta)\bs{\theta}(u_-+u_+)\right),
	\end{equation*}
	with $\delta_-(u_-,u_+,\beta)$ as given in
	\eqref{eq: delta_min scale func}. By Lemma~\ref{lemma: q vs Bc}, this
	becomes
	\begin{equation}\label{eq: tilde eta scale functions in theta}
		\tilde\eta(u_-, u_+, \alpha, \beta)
		= \frac{1}{p(\alpha)}\bs{q}(\alpha)^\top\bs{B}^{-1}
		\left(\bs{\theta}(u_+)
		- \delta_-(u_-,u_+,\beta)\bs{\theta}(u_-+u_+)\right).
	\end{equation}
	
	Conversely, with
	$\bs{M} \equiv \bs{M}(u_-,u_+,\beta)
	:= e^{\alpha_{d+1}(u_-+u_+)}\bs{B}-\tilde{\bs{B}}$,
	the martingale method gives
	\begin{align}
		\notag\tilde\eta(u_-, u_+, \alpha, \beta)
		&= \bs{\mathfrak{d}}(u_-, u_+)^\top
		\frac{\bs{q}(\alpha)}{p(\alpha)} \\
		&= \frac{1}{p(\alpha)}\bs{q}(\alpha)^\top\bs{M}^{-1}
		\left(e^{\alpha_{d+1}(u_-+u_+)}\bs{\theta}(u_+)
		- e^{\alpha_{d+1}u_+}\bs{\theta}(u_-+u_+)\right).
		\label{eq: tilde eta martingales in theta}
	\end{align}
	
	Thus, it remains to show that the expressions in
	\eqref{eq: tilde eta scale functions in theta} and
	\eqref{eq: tilde eta martingales in theta} coincide. For this, we first
	need the following identity, where $\bs{c}:=(c_1,\ldots,c_d)$.
	
	\begin{lemma}\label{lemma: relation B vs B_j}
		For $\beta>0$,
		\begin{equation}\label{eq: relation B vs B_j}
			\bs{B}^\top\bs{c}
			= -c_{d+1}\bs{\mathscr{B}}(\alpha_{d+1}).
		\end{equation}
	\end{lemma}
	
	\begin{proof}
		Let $\pi^*_1, \dots, \pi^*_{n^*}$ denote the roots of $p(\alpha)$ and
		let $d^*_1, \dots,d^*_{n^*}$ denote their multiplicities respectively.
		Then, for some constants $c^*_{i\ell j}$, we write
		$\mathscr{B}_j(\alpha)$ in terms of its partial fraction decomposition
		as follows:
		\begin{equation}\label{eq: partial frac decomp B_j(alpha)}
			\mathscr{B}_j(\alpha) = \frac{q_j(\alpha)}{p(\alpha)}
			= \sum_{i=1}^{n^*} \sum_{\ell = 1}^{d^*_i}
			\frac{c^*_{i\ell j}}{(\alpha - \pi^*_i)^\ell}.
		\end{equation}
		Furthermore, since the leading coefficient of $p(\alpha)$ is $1$, we
		can write
		\begin{equation*}
			p(\alpha) = \prod_{i=1}^{n^*}(\alpha - \pi^*_i)^{d^*_i},
		\end{equation*}
		from which it follows that, for $i=1,\dots,n^*$,
		\begin{equation}\label{eq: all deriv p zer0}
			p(\pi^*_i) = p^{(1)}(\pi_i^*) = \dots =
			p^{(d_i-1)}(\pi^*_i) = 0.
		\end{equation}
		Hence, for each $i=1,\dots,n^*$, the first $d_i^*-1$ derivatives of
		$p(\alpha)/\tilde{q}(\alpha)$ also vanish at $\alpha=\pi_i^*$. Using
		the partial fraction decomposition in
		\eqref{eq: partial frac decomp p / tilde q}, we obtain, for
		$\ell=0,\dots,d_i^*-1$,
		\begin{equation}\label{eq: all deriv p / tilde q zero}
			0 =
			\frac{\partial^\ell}{\partial \alpha^\ell}
			\frac{p(\alpha)}{\tilde{q}(\alpha)}
			\Bigg|_{\alpha = \pi^*_i}
			= \sum_{k=1}^{d+1}
			(-1)^\ell \ell!
			\frac{c_k}{(\pi^*_i - \alpha_k)^{\ell+1}}.
		\end{equation}
		Finally, using \eqref{eq: partial frac decomp B_j(alpha)}, we have
		\begin{align*}
			\sum_{k=1}^{d+1}c_k \mathscr{B}_j(\alpha_k)
			&= \sum_{k=1}^{d+1}c_k
			\sum_{i=1}^{n^*}\sum_{\ell=1}^{d^*_i}
			\frac{c^*_{i\ell j}}{(\alpha_k - \pi^*_i)^\ell} \\
			&= \sum_{i=1}^{n^*}\sum_{\ell=1}^{d^*_i}
			\frac{c^*_{i\ell j}}{(-1)^{\ell-1}(\ell-1)!}
			\sum_{k=1}^{d+1}
			(-1)^{\ell-1}(\ell-1)!
			\frac{c_k}{(\alpha_k - \pi^*_i)^\ell}
			= 0,
		\end{align*}
		by \eqref{eq: all deriv p / tilde q zero}. Thus, for all $j\in[d]$,
		\begin{equation*}
			\sum_{k=1}^{d}c_k\mathscr{B}_j(\alpha_k)
			= -c_{d+1}\mathscr{B}_j(\alpha_{d+1}),
		\end{equation*}
		which is equivalent to \eqref{eq: relation B vs B_j}.
	\end{proof}
	
	The preceding identity allows us to establish the required equality
	between \eqref{eq: tilde eta scale functions in theta} and
	\eqref{eq: tilde eta martingales in theta}. More precisely, it implies
	the following lemma.
	
	\begin{lemma}\label{lemma: relation M B_inv}
		For $u_-,u_+\geqslant0$ with $u_-+u_+ >0$, and $\beta > 0$,
		\begin{equation}\label{eq: MB_inv}
			\bs{M}\bs{B}^{-1}
			\left(\bs{\theta}(u_+)
			- \delta_-(u_-,u_+,\beta)\bs{\theta}(u_-+u_+)\right)
			= e^{\alpha_{d+1}(u_-+u_+)}\bs{\theta}(u_+)
			- e^{\alpha_{d+1}u_+}\bs{\theta}(u_-+u_+).
		\end{equation}
	\end{lemma}
	
	\begin{proof}
		Starting with
		$\bs{M}\bs{B}^{-1}
		=e^{\alpha_{d+1}(u_-+u_+)}\bs{I}
		-\tilde{\bs{B}}\bs{B}^{-1}$,
		we obtain for the left-hand side of \eqref{eq: MB_inv}
		\begin{align}
			\notag e^{\alpha_{d+1}(u_-+u_+)}
			&\bs{\theta}(u_+)
			-e^{\alpha_{d+1}(u_-+u_+)}
			\delta_-(u_-,u_+,\beta)\bs{\theta}(u_-+u_+) \\
			&-\tilde{\bs{B}}\bs{B}^{-1}\bs{\theta}(u_+)
			+\delta_-(u_-,u_+,\beta)
			\tilde{\bs{B}}\bs{B}^{-1}\bs{\theta}(u_-+u_+).
			\label{eq: temp MB_inv}
		\end{align}
		Furthermore, by Lemma~\ref{lemma: relation B vs B_j},
		\begin{equation*}
			\tilde{\bs{B}}\bs{B}^{-1}
			=\bs{\theta}(u_-+u_+)
			\left(\bs{\mathscr{B}}(\alpha_{d+1})^\top\bs{B}^{-1}\right)
			=-\frac{1}{c_{d+1}}
			\bs{\theta}(u_-+u_+)\bs{c}^\top.
		\end{equation*}
		Hence \eqref{eq: temp MB_inv} can be written as
		\begin{align}
			\notag e^{\alpha_{d+1}(u_-+u_+)}
			&\bs{\theta}(u_+)
			-e^{\alpha_{d+1}(u_-+u_+)}
			\delta_-(u_-,u_+,\beta)\bs{\theta}(u_-+u_+) \\
			&\:\:\:\:\:
			+\frac{1}{c_{d+1}}
			\left(\bs{c}^\top\bs{\theta}(u_+)\right)
			\bs{\theta}(u_-+u_+)
			-\frac{1}{c_{d+1}}
			\delta_-(u_-,u_+,\beta)
			\left(\bs{c}^\top\bs{\theta}(u_-+u_+)\right)
			\bs{\theta}(u_-+u_+).
		\end{align}
		This can be rewritten as
		\begin{align*}
			&e^{\alpha_{d+1}(u_-+u_+)}\bs{\theta}(u_+)
			-\frac{1}{c_{d+1}}
			\left(
			c_{d+1}e^{\alpha_{d+1}(u_-+u_+)}
			+\bs{c}^\top\bs{\theta}(u_-+u_+)
			\right)
			\delta_-(u_-,u_+,\beta)\bs{\theta}(u_-+u_+)\\
			&\quad
			+\frac{1}{c_{d+1}}
			\left(\bs{c}^\top\bs{\theta}(u_+)\right)
			\bs{\theta}(u_-+u_+).
		\end{align*}
		With the expression for $\delta_-(u_-,u_+,\beta)$ in
		\eqref{eq: delta_min scale func}, this simplifies to
		\begin{align*}
			&e^{\alpha_{d+1}(u_-+u_+)}\bs{\theta}(u_+)
			-\frac{1}{c_{d+1}}
			\left(\sum_{k=1}^{d+1}c_ke^{\alpha_k u_+}\right)
			\bs{\theta}(u_-+u_+)
			+\frac{1}{c_{d+1}}
			\left(\bs{c}^\top\bs{\theta}(u_+)\right)
			\bs{\theta}(u_-+u_+),
		\end{align*}
		which equals the right-hand side of \eqref{eq: MB_inv}.
	\end{proof}

	\section{Steady-state derivations}
	\subsection{Proof of Proposition~\ref{theorem: stationary LST workload}}\label{app: proof thrm stationary LST}
	Recall that we aim to derive
	\begin{equation*}
		\lim_{\beta\downarrow0} \chi(0,\alpha,\beta) = \lim_{\beta\downarrow0} \frac{(1-P_{22})b_1+P_{12}b_2}
		{\det(\bs{I}-\bs{P})}.
	\end{equation*}
	
	\subsubsection*{Numerator}
	Since $b_1$ and $b_2$ vanish at $\beta=0$, their first-order terms determine the first-order term of the numerator. Dividing \eqref{eq: def b_1} and \eqref{eq: def b_2} by $\beta$ and letting $\beta\downarrow0$, followed by basic simplifications, yields
	\begin{align*}
		\lim_{\beta\downarrow0}\frac{b_1(v,\alpha,\beta)}{\beta}
		&=
		\frac{1}{\lambda_1}
		-\frac{1}{\varphi_2(\alpha)}
		\mathbb{E}\left(e^{-\alpha B_1}\right)
		+\frac{e^{-\alpha v}}{\varphi_2(\alpha)}
		\\
		&\quad
		+\left(
		\frac{\varphi_1(\alpha)}{\varphi_2(\alpha)}-1
		\right)
		\int_0^v e^{-\alpha x}
		\int_0^{v-x}e^{-\alpha y}W_1^{(0)}(y)
		\,\mathrm{d}y\,
		\mathbb{P}(B_1\in\mathrm{d}x)
		\\
		&\quad
		+\frac{1-e^{-\alpha v}}{\varphi_2(\alpha)}
		\int_0^v
		\frac{W_1^{(0)}(v-x)}{W_1^{(0)}(v)}
		\mathbb{P}(B_1\in\mathrm{d}x)
		\\
		&\quad
		-\left(
		\frac{\varphi_1(\alpha)}{\varphi_2(\alpha)}-1
		\right)
		\int_0^v e^{-\alpha y}W_1^{(0)}(y)\,\mathrm{d}y
		\int_0^v
		\frac{W_1^{(0)}(v-x)}{W_1^{(0)}(v)}
		\mathbb{P}(B_1\in\mathrm{d}x),\\
		\lim_{\beta\downarrow0}\frac{b_2(v,\alpha,\beta)}{\beta}
		&=
		-\frac{e^{-\alpha v}}{\varphi_2(\alpha)}
		\frac{W_1^{(0)}(0)}{W_1^{(0)}(v)}
		\left(
		1-e^{\alpha v}
		\left(
		1-\varphi_1(\alpha)
		\int_0^v e^{-\alpha y}W_1^{(0)}(y)\,\mathrm{d}y
		\right)
		\right)
		\\
		&\quad
		+\frac{W_1^{(0)}(0)}{W_1^{(0)}(v)}
		\int_0^v e^{-\alpha y}W_1^{(0)}(y)\,\mathrm{d}y.
	\end{align*}
	At $\beta=0$, using \eqref{eq: def P_{12}} and \eqref{eq: def P_{2}},
	\begin{align*}
		P_{12}(v,0)
		&=
		1-
		\int_{[0,v]}
		\frac{W_1^{(0)}(v-x)}{W_1^{(0)}(v)}
		\mathbb{P}(B_1\in\mathrm{d}x),\quad\quad
		P_{22}(v,0)
		=
		1-
		\frac{W_1^{(0)}(0)}{W_1^{(0)}(v)}.
	\end{align*}
	Consequently,
	\begin{align}
		(1-P_{22})b_1+P_{12}b_2
		&=
		\frac{W_1^{(0)}(0)}{W_1^{(0)}(v)}
		\Bigg[
		\frac{1}{\lambda_1}
		+\frac{1-\mathscr{B}_1(\alpha)}{\varphi_2(\alpha)}
		\notag\\
		&\hspace{-1cm}
		+\left(
		1-\frac{\varphi_1(\alpha)}{\varphi_2(\alpha)}
		\right)
		\int_0^v e^{-\alpha y}W_1^{(0)}(y)
		\left(
		1-
		\int_0^{v-y}e^{-\alpha x}
		\mathbb{P}(B_1\in\mathrm{d}x)
		\right)\mathrm{d}y
		\Bigg]\beta
		+O(\beta^2).
		\label{eq: taylor exp numer}
	\end{align}
	
	\subsubsection*{Denominator}
	
	The denominator is analyzed along the same lines. Substituting the expressions for the entries of $\bs{P}$, as given in \eqref{eq: def P_{11}}, \eqref{eq: def P_{12}} and \eqref{eq: def P_{2}}, in $\det(\bs{I}-\bs{P})$ and simplifying gives
	\begin{align*}
		\det(\bs{I}-\bs{P})
		&=
		\frac{W_1^{(\beta)}(0)}{W_1^{(\beta)}(v)}
		e^{\psi_2(\beta)v}
		\left(
		1-
		(\varphi_1(\psi_2(\beta))-\beta)
		\int_0^v e^{-\psi_2(\beta)y}
		W_1^{(\beta)}(y)\,\mathrm{d}y
		\right)
		\\
		&\quad
		+\frac{\lambda_1}{\lambda_1+\beta}
		\frac{W_1^{(\beta)}(0)}{W_1^{(\beta)}(v)}
		\Bigg[
		(\varphi_1(\psi_2(\beta))-\beta)
		\int_0^v e^{\psi_2(\beta)(v-x)}
		\\
		&\qquad\qquad\qquad\qquad\qquad
		\times
		\int_0^{v-x}e^{-\psi_2(\beta)y}
		W_1^{(\beta)}(y)\,\mathrm{d}y\,
		\mathbb{P}(B_1\in\mathrm{d}x)
		-
		\mathbb{E}\left(
		e^{-\psi_2(\beta)(B_1-v)}
		\right)
		\Bigg].
	\end{align*}
	Differentiating with respect to $\beta$ at $\beta=0$ and using Assumption \ref{assump: stability}, we obtain that, up to $O(\beta^2)$-terms,
	\begin{align}
		\det(\bs{I}-\bs{P})
		&=
		\frac{W_1^{(0)}(0)}{W_1^{(0)}(v)}
		\Bigg[
		\frac{1}{\lambda_1}
		+\psi_2'(0)\mathbb{E}(B_1)
		+\left(
		1-\varphi_1'(0)\psi_2'(0)
		\right)
		\int_0^v
		W_1^{(0)}(y)
		\mathbb{P}(B_1>v-y)\,\mathrm{d}y
		\Bigg]\,\beta.
		\label{eq: taylor exp denom}
	\end{align}
	
	Since
	\begin{equation*}
		\frac{W_1^{(0)}(0)}{W_1^{(0)}(v)}
		=
		\delta_{1,-}(v,0,0)
		=
		\mathbb{P}(\sigma_2(v)<\tau_2(0))
		>0,
	\end{equation*}
	we may divide \eqref{eq: taylor exp numer} and \eqref{eq: taylor exp denom} by
	$\beta W_1^{(0)}(0)/W_1^{(0)}(v)$ and let $\beta\downarrow0$. We obtain
	\begin{equation*}
		\lim_{\beta\downarrow0}\chi(0,\alpha,\beta) = \frac{f(v,\alpha)}{f(v,0)}
	\end{equation*}
	for
	\begin{align*}
		f(v,\alpha) :=\,& \frac{1}{\lambda_1}
		+\frac{1-\mathscr{B}_1(\alpha)}{\varphi_2(\alpha)}
		+\left(
		1-\frac{\varphi_1(\alpha)}{\varphi_2(\alpha)}
		\right)
		\int_0^v e^{-\alpha y}W_1^{(0)}(y)
		\left(
		1-
		\int_0^{v-y}e^{-\alpha x}
		\mathbb{P}(B_1\in\mathrm{d}x)
		\right)\mathrm{d}y
	\end{align*}
	and
	\begin{equation*}
		f(v,0) = \lim_{\alpha\downarrow0}f(v,\alpha) = \frac{1}{\lambda_1}
		+\psi_2'(0)\mathbb{E}(B_1)
		+\left(
		1-\varphi_1'(0)\psi_2'(0)
		\right)
		\int_0^v
		W_1^{(0)}(y)
		\mathbb{P}(B_1>v-y)\,\mathrm{d}y,
	\end{equation*}
	thus proving Proposition~\ref{theorem: stationary LST workload}. \hfill $\square$
	
	\subsection{Proofs of \texorpdfstring{\eqref{eq: limit_integral_I}}{} and \texorpdfstring{\eqref{eq: limit_integral_II}}{}}\label{app: limit_integrals}
	For the integral in \eqref{eq: limit_integral_I}, the domain of integration is
	\[
	\{(x,y)\in\mathbb{R}^2:x,y\geqslant0,\ x+y\leqslant v\}.
	\]
	As $v\to\infty$, this domain increases to $\{(x,y)\in\mathbb{R}^2:x,y\geqslant0\}$. Hence,
	\begin{align*}
		\lim_{v\to\infty}
		&\int_0^v e^{-\alpha x}
		\int_0^{v-x}e^{-\alpha y}W_1^{(0)}(y)\,\mathrm{d}y\,
		\mathbb{P}(B_1\in\mathrm{d}x)\\
		&=
		\int_0^\infty e^{-\alpha x}
		\int_0^\infty e^{-\alpha y}W_1^{(0)}(y)\,\mathrm{d}y\,
		\mathbb{P}(B_1\in\mathrm{d}x)
		=
		\frac{\mathscr{B}_1(\alpha)}{\varphi_1(\alpha)}.
	\end{align*}
	
	For \eqref{eq: limit_integral_II} we use the result of the first equality. We rewrite the limit as follows:
	\begin{align*}
		\lim_{v\to\infty}\int_0^v&W_1^{(0)}(y)\,\mathbb{P}(B_1> v-y)\,\mathrm{d}y \\
		&= \lim_{v\to\infty} \left( \int_0^v W_1^{(0)}(y)\,\mathrm{d}y - \int_0^v W_1^{(0)}(y)\,\int_0^{v-y} \mathbb{P}(B_1\in\mathrm{d}x)\,\mathrm{d}y \right) \\
		&= \lim_{v\to\infty} \left( \int_0^v W_1^{(0)}(y)\,\mathrm{d}y - \int_0^v \,\int_0^{v-x} W_1^{(0)}(y) \mathrm{d}y\,\mathbb{P}(B_1\in\mathrm{d}x) \right) \\
		&= \lim_{v\to\infty} \lim_{\alpha\downarrow0} \left( \int_0^v e^{-\alpha y} W_1^{(0)}(y)\,\mathrm{d}y - \int_0^v e^{-\alpha x}\,\int_0^{v-x} e^{-\alpha y}W_1^{(0)}(y) \mathrm{d}y\,\mathbb{P}(B_1\in\mathrm{d}x) \right).
	\end{align*}
	By the monotone convergence theorem we can interchange the order of the limits, so that, by \eqref{eq: limit_integral_I}, the equation above equals
	\begin{align*}
		\lim_{\alpha\downarrow0}\frac{1}{\varphi_1(\alpha)} -\frac{\mathscr{B}_1(\alpha)}{\varphi_1(\alpha)}.
	\end{align*}
	By l'Hôpital's rule this is equal to \eqref{eq: limit_integral_II}. \hfill $\square$
	
	\subsection{Proof that the LSTs equal}\label{app: proof equivalence LSTs}
	
	In \cite[Subsection 4.2]{BBR} it is shown that the LST of the steady-state workload is given by
	\begin{equation}\label{eq: LST steady-state BBR}
		\mathbb{E}e^{-\alpha V_v} = C\left(\frac{\alpha}{\varphi_2(\alpha)} + \left(1 - \frac{\varphi_1(\alpha)}{\varphi_2(\alpha)}\right)\left(W^{(0)}_1(0) + \int_0^v e^{-\alpha x}\mathrm{d}W^{(0)}_1(x)\right)\right),
	\end{equation}
	with $\rho_i := {\lambda_i\mathbb{E}(B_i)}/{r_i^*}$ and drift parameter $r_i^* = -r_i > 0$,
	\begin{align*}
		C := \frac{r^*_2(1-\rho_2)}{1+\left(r^*_2(1-\rho_2) - r^*_1(1-\rho_1)\right)W_1^{(0)}(v)}.
	\end{align*}
	
	We require the following identities in order to prove the equivalence of our expression for the LST in Proposition~\ref{theorem: stationary LST workload}, and the expression in \eqref{eq: LST steady-state BBR}.
	\begin{lemma}\label{lemma: proof LSTs same}
		For $v>0$,
		\begin{equation}\label{eq: proof LSTs same}
			\int_0^v e^{-\alpha x}\mathrm{d} W^{(0)}_1(x) = \frac{\lambda_1}{-r_1} \int_0^v e^{-\alpha y} W^{(0)}_1(y) \left(1- \int_0^{v-y} e^{-\alpha x} \mathbb{P}(B_1\in\mathrm{d}x)\right)\,\mathrm{d}y.
		\end{equation}
	\end{lemma}
	\begin{proof}
		We show the equality by proving that the Laplace transforms with respect to $v$ of both sides are equal. Multiplying the left-hand side of \eqref{eq: proof LSTs same} by $e^{-\omega v}$, integrating $v$ over $(0,\infty)$ and changing the order of integration yields 
		\begin{align*}
			\int_0^\infty e^{-\alpha x} \left( \int_x^\infty e^{-\omega v}\,\mathrm{d}v\right)\mathrm{d}W^{(0)}_1(x) = \frac{1}{\omega}\int_0^\infty e^{-(\alpha+\omega) x} \mathrm{d} W^{(0)}_1(x).
		\end{align*}
		Integrating by parts then gives
		\begin{equation}\label{eq: proof LSTs same step 1}
			\frac{1}{\omega}\left(-W^{(0)}_1(0) + \frac{\alpha+\omega}{\varphi_1(\alpha+\omega)}\right).
		\end{equation}
		Multiplying the right-hand side of \eqref{eq: proof LSTs same} by $e^{-\omega v}$, integrating $v$ over $(0,\infty)$ and changing the order of integration yields for the first term
		\begin{equation}\label{eq: proof LSTs same step 2a}
			\frac{\lambda_1}{-r_1}\int_0^\infty e^{-\alpha y} \left( \int_y^\infty e^{-\omega v}\,\mathrm{d}v\right) W^{(0)}_1(y) \,\mathrm{d}y =\frac{\lambda_1}{-r_1\omega} \frac{1}{\varphi_1(\alpha + \omega)}.
		\end{equation}
		For the second, double integral, after applying the change in the order of integration $(\mathrm{d}x, \mathrm{d}y, \mathrm{d}v) \to (\mathrm{d}x, \mathrm{d}v, \mathrm{d}y)$, substituting $z:= v-y$, and then changing the order of integration to $(\mathrm{d}x, \mathrm{d}z, \mathrm{d}y) \to (\mathrm{d}z, \mathrm{d}x, \mathrm{d}y)$, we obtain
		\begin{align*}
			\frac{\lambda_1}{-r_1}\int_0^\infty e^{-(\alpha+\omega)y}W^{(0)}_1(y) \int_0^\infty e^{-\alpha x} \int_x^\infty e^{-\omega z}\,\mathrm{d}z\,\mathbb{P}(B_1\in\mathrm{d}x)\,\mathrm{d}y,
		\end{align*}
		which can be simplified to
		\begin{equation}\label{eq: proof LSTs same step 2b}
			\frac{\lambda_1}{-r_1\omega}\frac{\mathscr{B}_1(\alpha+\omega)}{\varphi_1(\alpha+\omega)}.
		\end{equation}
		Taking the difference of \eqref{eq: proof LSTs same step 2a} and \eqref{eq: proof LSTs same step 2b} gives for the Laplace transform of the right-hand side of \eqref{eq: proof LSTs same}
		\begin{align*}
			\frac{\lambda_1}{-r_1\omega} \frac{1-\mathscr{B}_1(\alpha+\omega)}{\varphi_1(\alpha+\omega)}.
		\end{align*}
		Using \eqref{eq: def laplace exp}, this can be rewritten to
		\begin{equation*}
			\frac{1}{\omega}\left(\frac{1}{r_1} + \frac{\alpha+\omega}{\varphi_1(\alpha+\omega)}\right).
		\end{equation*}
		Since $W_1^{(0)}(0) = -r_1^{-1}$ for a compound Poisson process with drift, this is equal to \eqref{eq: proof LSTs same step 1}, thus showing the equivalence of \eqref{eq: proof LSTs same}.
	\end{proof}
	Taking $\alpha=0$ in \eqref{eq: proof LSTs same} yields the following renewal equation; cf.\ the level-crossing based identities in \cite{DO}.
	\begin{corollary}\label{coro: renewal equation}
		For $x\in[0,\infty)$,
		\begin{equation}\label{eq: renewal equation}
			W_1^{(0)}(x) = -\frac{1}{r_1} - \frac{\lambda_1}{r_1}\int_0^x W_1^{(0)}(y)\mathbb{P}(B_1>x-y)\,\mathrm{d}y.
		\end{equation}
	\end{corollary}
	
	Using Lemma \ref{lemma: proof LSTs same}, it is easily verified that the expression within the outermost brackets of the right-hand side of \eqref{eq: LST steady-state BBR} is equal to $\tfrac{\lambda_1}{-r_1}f(v,\alpha)$, and with Corollary \ref{coro: renewal equation} one can show that $C = \tfrac{-r_1}{\lambda_1}f(v,0)^{-1}$, thus showing that the expressions are equivalent.

\end{document}